\documentclass[11pt]{article}
\usepackage{graphicx} 
\usepackage{tikz} 
\usepackage{tikz-cd} 
\usepackage{amsthm} 
\usepackage{amsmath, amsfonts, amssymb,  mathtools} 
\usepackage{pinlabel}
\usepackage{subcaption}
\usepackage{url}

\newcommand{\Addresses}{{
  \bigskip
  \footnotesize
  
  \textsc{Department of Mathematics, Georgia Institute of Technology,
    Atlanta, Georgia 30332}\par\nopagebreak
  \textit{E-mail address}: \texttt{seaneli@gatech.edu}

}}

\theoremstyle{plain}
\newtheorem{theorem}{Theorem}[section]
\newtheorem{conjecture}{Conjecture}[theorem]
\newtheorem{corollary}[theorem]{Corollary}
\newtheorem{lemma}[theorem]{Lemma}
\newtheorem{proposition}[theorem]{Proposition}
\newtheorem{question}[theorem]{Question}
\theoremstyle{definition}
\newtheorem{definition}[theorem]{Definition}

\newtheorem*{note}{Note}
\newtheorem{remark}[theorem]{Remark}

\DeclarePairedDelimiter\abs{\lvert}{\rvert}%
\DeclarePairedDelimiter\norm{\lVert}{\rVert}%
\DeclarePairedDelimiter\set{\{}{\}}%
\makeatletter
\let\oldabs\abs
\def\abs{\@ifstar{\oldabs}{\oldabs*}}
\let\oldnorm\norm
\def\norm{\@ifstar{\oldnorm}{\oldnorm*}}
\let\oldset\set
\def\set{\@ifstar{\oldset}{\oldset*}}
\makeatother

\newcommand{\ve}{\varepsilon}

\DeclareMathOperator{\coker}{coker}

\newcommand{\into}{\hookrightarrow}

\newcommand{\A}{\mathcal{A}}

\newcommand{\cF}{\mathcal{F}}

\newcommand{\cK}{\mathcal{K}}

\newcommand{\M}{\mathcal{M}}

\newcommand{\cR}{{\mathcal{R}}}

\newcommand{\ff}{\mathfrak{f}}

\newcommand{\R}{\mathbb{R}}

\newcommand{\bbC}{{\mathbb C}}

\newcommand{\bbN}{{\mathbb N}}

\newcommand{\bbR}{{\mathbb R}}

\newcommand{\bbZ}{{\mathbb Z}}

\def\HF {\mathit{HF}}

\newcommand\HFp {\HF^+}

\newcommand \HFm {\HF^-}

\newcommand \HFinf {\HF^{\infty}}

\def\HFred{\HF_{\operatorname{red}}}

\def\s{\mathfrak s}
\def\t{\mathfrak t}
\def\ff{\mathbb{F}}
\newcommand\HE{\mathit{HE}}

\def\coker{\operatorname{coker}}

\usepackage{overpic}
\title{Irreducible proper 2-knots from exotic open 2-handles
}
\author{Sean Eli}
\date{}
\begin{document}
\maketitle

\begin{abstract}
We construct infinite families of irreducible exotic proper knotted surfaces in $\R^4$, making progress on a question of Gompf. Here \textit{irreducible} means these surfaces are not end-sums of standard surfaces with exotic planes. To prove the topological equivalence, we give a highly flexible construction of exotic open 2-handles, which generalizes several similar constructions in the literature. We distinguish exotic surfaces through the genus functions and end Floer homology of their double branched covers. By studying these generalized handles further, we construct a new family of topologically slice links.
\end{abstract}

\section{Introduction} Exotically knotted surfaces in 4-manifolds have recently seen intense study. In \cite{gompfproper}, Gompf studies exotic \textit{proper 2-knots}: proper embeddings of noncompact surfaces into $\R^4$ that are topologically but not smoothly ambiently isotopic. One can often change the smooth type of a proper 2-knot by \textit{end-summing} with a topologically trivial exotic plane, which does not change the topological type. Question~4.5 of \cite{gompfproper} essentially asks, given a topological type of a proper 2-knot in $\R^4$, how many smooth proper 2-knots there are of this type, modulo end-summing with exotic planes. Our main result is the following.

\begin{theorem}\label{thm:mobius}
    There is an explicit infinite family of topologically standard exotic proper open M\"obius strips (resp.~open annuli) in $\R^4$, which are branch loci of double branched covering maps from distinct exotic smoothings of $\overline{\bbC P^2}\setminus pt$ (resp.~$S^2\times \R^2$). These are irreducible, i.e.~not end sums of smoothly standard surfaces with {any} exotic planes, and stay smoothly nonisotopic after end-summing with any {small} exotic planes. Moreover, letting $A_n$ denote the annuli, $A_n\natural P$ is not isotopic to $A_k$ for any $n,k$ and {any} exotic plane $P$.
\end{theorem}

 One of the M\"obius strips is shown in Figure~\ref{fig:mobiusfront}. Here, a \textit{small exotic plane} is one whose branched double cover is a small exotic $\R^4$, i.e. one whose compact codimension-0 submanifolds smoothly embed in $\R^4$. Our M\"obius strips and open annuli are distinct from Gompf's \cite[Theorem 4.2]{gompfproper}, which are end-sums of standard surfaces with exotic planes. We distinguish them using the genus function of their double branched covers. A key ingredient of Theorem~\ref{thm:mobius} is a new, flexible construction of exotic open 2-handles. These 2-handles can be made to satisfy many properties, for example the exotic open 2-handles $BH_m$ used to prove Theorem \ref{thm:mobius} are Stein, symmetric, and have topologically standard branch sets. 
 To prove these open 2-handles are topologically standard, we standardize a large class of `generalized Casson handles', defined by Casson's criteria for having the right proper homotopy type. This class includes all Casson handles, a generalization called \textit{triangular Casson handles} (defined in Section \ref{sec:handle}), Gompf's generalized Casson handles \cite{gompfsteinisotopy}, and all \textit{open infinite towers} of disk embedding theory \cite{diskembedding, freedmanquinn}. 
\begin{theorem}\label{thm:open2handles}
       Any generalized Casson handle $(V,\partial V)$ is homeomorphic to an open 2-handle $(D^2\times \bbR^2, S^1\times \bbR^2)$, and therefore has a topologically flat core disk bounded by its attaching circle.
\end{theorem}
We prove Theorem \ref{thm:open2handles} by constructing a proper $h$-cobordism to the standard open 2-handle, rel boundary, and invoking the fundamental group $\bbZ$ at infinity proper $h$-cobordism theorem due to Freedman-Quinn  
\cite[Corollary 7.3B]{freedmanquinn}. A similar Theorem is stated in Freedman's ICM address \cite[Theorem 4]{freedmanICM}, and claimed to be a consequence of a proper $s$-cobordism theorem with fundamental group $\bbZ$ at infinity, but a proof was not given. In particular, it was not shown how to obtain the proper $h$-cobordism, and we are not aware of a proof in the literature. We give a detailed proof of Theorem~\ref{thm:open2handles} in Section~\ref{sec:embeddings}. 

Theorem \ref{thm:open2handles} has the following applications. First, it allows us to construct noncompact exotica by swapping 2-handles for generalized Casson handles, detectable with the genus function \cite{gompfgenera} or end Floer homology \cite{elihomlidman, gadgil}. Second, it establishes the existence of topologically flat core disks in many settings other than just Casson handles and skyscrapers \cite{diskembedding}. Recall that \textit{infinite towers} \cite[Definition 12.16]{diskembedding}, \cite{freedmanquinn} are generalizations of Casson handles better adapted to non-simply connected situations: while a Casson handle is an infinite stack of stages of thickened immersed disks, infinite towers consist of \textit{stories}, which in turn have stages of thickened embedded compact surfaces and immersed disks. In \cite{diskembedding} and \cite{freedmanquinn}, infinite towers with at least four surface stages per story, and exponentially growing numbers of surface stages (called \textit{skyscrapers} or \textit{convergent infinite towers}), are shown to compactify to manifolds homeomorphic to the closed 2-handle, and hence have flat core disks. 
As a consequence of Theorem \ref{thm:open2handles}, we find \textit{all} open (i.e. with all boundary deleted except the attaching region) infinite towers are homeomorphic to open 2-handles, and thus have flat core disks bounded by their attaching circles. 
\begin{corollary}\label{cor:skyscrapers}
     All open infinite towers are homeomorphic to open 2-handles, and hence have topologically flat core disks.
 \end{corollary}
In particular, this applies to the simplest infinite tower with surface stages, which can be seen in Figure~\ref{fig:exoticR4skyscraper}. Moreover, the same holds for more general hybrids of infinite towers and Casson handles, modified by triangular families of curves (e.g. the $BH_m$ of Section \ref{sec:handle}). While the problem of explicitly describing a topologically flat core disk is notoriously difficult (see \cite{chapowell} and Problem 5.4 of the AIM problem list \cite{AIMlist}), it might be easier to find these disks in a suitable general setting. 
\begin{figure}[t!]
\centering
\begin{overpic}[width=0.8\textwidth, 
 unit=1mm, tics=5]{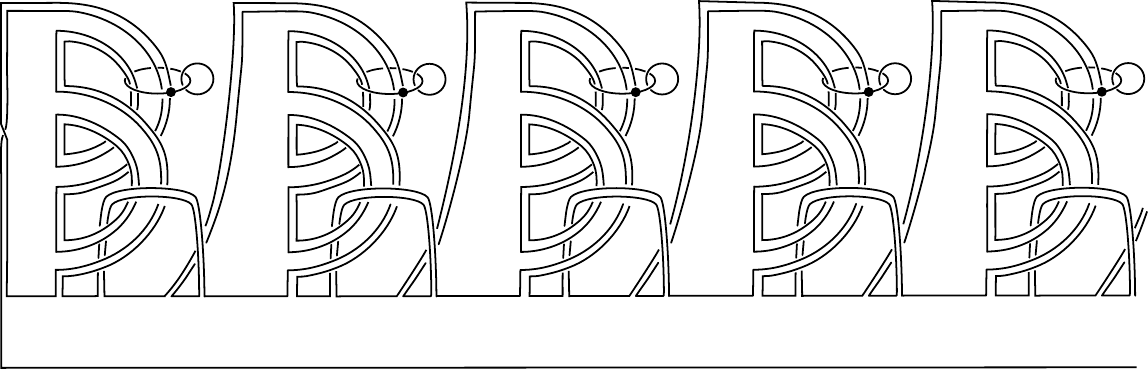}
\put(15, 27){$CH$} 
\put(35.5, 27){$CH$}
\put(55.5, 27){$CH$}
 \put(76, 27){$CH$}
 \put(96, 27){$CH$}
 \put(99, 15){$\hdots$}
  \put(99, 3){$\hdots$}
 \put(17, 21.7){$0$} 
 \put(37, 21.7){$0$} 
 \put(57, 21.7){$0$} 
 \put(77.5, 21.7){$0$} 
 \put(98, 21.7){$0$} 
 \end{overpic}
\caption{Exotic M\"obius strip in the standard $\R^4$. All Casson handles are $CH_+$. Figures~\ref{fig:branchset2} and \ref{fig:branchset3} show that replacing the Casson handles with 2-handles yields a standard M\"obius strip.}
\label{fig:mobiusfront}
\end{figure}
\begin{figure}[ht!]
\centering
\begin{overpic}[width=0.4\textwidth, 
 unit=1mm, tics=5]{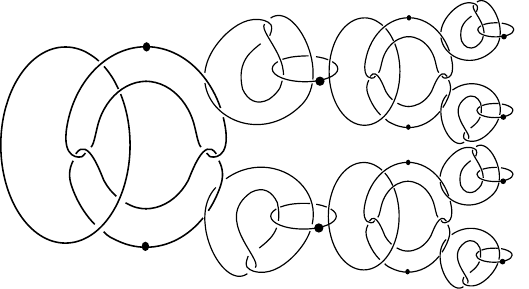}
\put(38, 2){0} 
\put(38, 47){0} 
\put(60, 2){0}
\put(60, 47){0}
 \put(99, 27){$\hdots$}
  \put(94, -1){$0$}
  \put(94, 15){$0$}
  \put(94, 37){$0$}
   \put(94, 52){$0$}
 \end{overpic}
\caption{The simplest infinite tower having surface stages.  }
\label{fig:exoticR4skyscraper}
\end{figure}
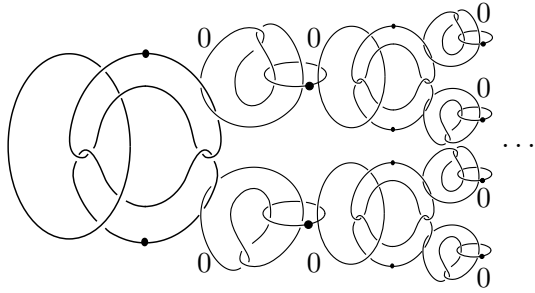
 
Generalized Casson handles also appear in the context of exotic planes. To see how, recall in the recent work \cite{endkhovanov} Teng constructed an exotic plane in $\R^4$ \textit{without} showing its double branched cover $\cR$ is an exotic $\R^4$. Note that it is an open question whether the double branched cover of an exotic plane in $\R^4$ can be the standard $\R^4$ \cite[Problem 6.3]{gompfproper}. Using the algorithm of Akbulut-Kirby \cite{akbulutkirby} we obtain the Kirby diagram for $\cR$ shown in Figure~\ref{fig:exoticR4}.
\begin{figure}[ht]
\centering
\begin{subfigure}{0.3\textwidth}
            \begin{overpic}[width=\textwidth, 
 unit=1mm, tics=5]{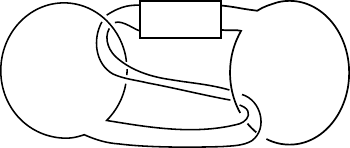}
         \put(45, 35){$-3$}
            \end{overpic}
            \caption{}
            \label{fig:12n582}
        \end{subfigure} 
        \hspace{5mm}
        \begin{subfigure}{0.6\textwidth}
            \begin{overpic}[width=\textwidth, 
 unit=1mm, tics=5]{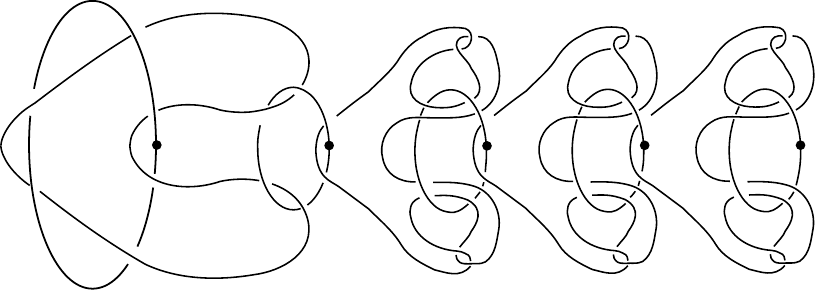}
\put(36, 32){$-1$} 
\put(60, 32){$0$} 
\put(79, 32){$0$} 
\put(99, 32){$0$} 
\put(100, 17.5){$\hdots$} 
 \end{overpic}
 \caption{}
 \label{fig:exoticR4}
        \end{subfigure}
\caption{(a): Ribbon knot 12n582. (b): A Kirby diagram for $\cR$. }
\end{figure}
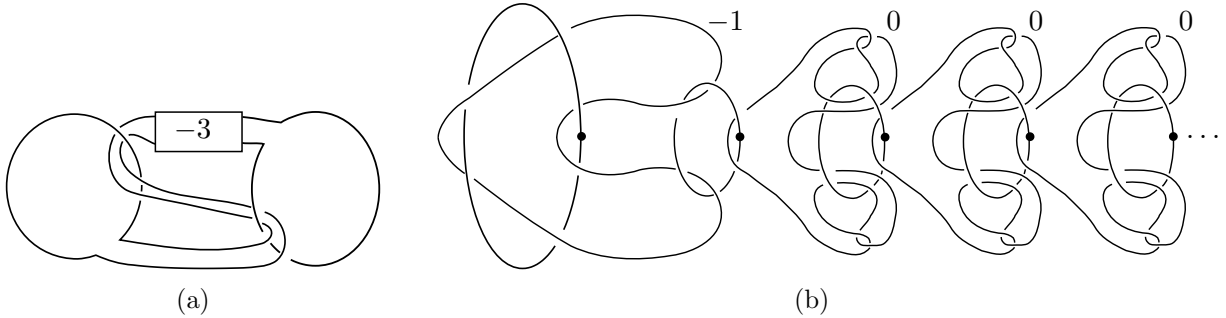
This manifold is similar to a ribbon $\bbR^4$ \cite{bizacagompf, freedmandemichelis}: note the two leftmost 1-handles and the $-1$-framed 2-handle form a ribbon disk complement for the knot 12n582 (see Knotinfo \cite{knotinfo}) shown in Figure \ref{fig:12n582}. The infinite periodic piece resembles a Casson handle; it is obtained by modifying an infinite tower of self-plumbed 2-handles using a \textit{triangular family of curves}, discussed in Section~\ref{sec:handle}. 
We show the following:
\begin{theorem}\label{theorem:exoticR4}
    The manifold $\cR$ is a Stein exotic $\bbR^4$ that embeds in the standard $\bbR^4$. Moreover $\cR$, $\overline{\cR}$ and their end sum $\cR\natural \overline{\cR}$ are three distinct exotic $\R^4$'s with double branched covering actions, hence their branch sets are distinct exotic planes.
\end{theorem}
The genus function cannot directly show $\cR$ is exotic. Instead, we give a short proof of exoticness with end Floer homology \cite{elihomlidman, gadgil} combined with the $\bbC P^2$-stable diffeomorphism technique of Bi\v zaca-Gompf \cite{bizacagompf}. Since exoticness of the branched cover $\cR$ implies exoticness of Teng's plane, we have reproven Teng's Theorem A. 
Thus, end Floer homology is a relevant tool to study Question 6.13 of Gompf, asking for direct invariants to detect exotic planes \cite{gompfproper}. We discuss the problem of detecting infinite families based on $\cR$ in Section~\ref{sec:exotic}; changing the $-1$-framing in Figure \ref{fig:exoticR4} appears likely to produce a family with isomorphic end Floer homology, hence more subtle invariants may be needed. 
Turning our attention to the infinite periodic piece of $\cR$, which we call $(TH, \partial_-TH)$ throughout, we find:  

\begin{theorem}\label{thm:TH}
    The manifold $(TH, \partial_-TH)$ is a Stein exotic open 2-handle. Its attaching curve has smooth slice genus 2 in $TH$, and $TH$ has a topologically flat core disk inside its first four stages.
\end{theorem}

We determine the slice genus using a Stein adjunction inequality. Consequently, $TH$ is not diffeomorphic to any member of the family of exotic open 2-handles considered by Gompf in \cite[Question 6.12]{gompfproper}, which have smooth core disks. It is possible that $TH$ is diffeomorphic to a Casson handle, and we discuss in Section~\ref{sec:stein}. 

To obtain the four-stage topologically flat core disk, we adapt a technique of Cha and Powell \cite{chapowell} to a specialized family of generalized Casson handles, which includes $TH$. Following \cite{chapowell}, we obtain a family of topologically slice links. Cha and Powell define the \textit{ramified Whitehead doubling} operations: first take 0-framed pushoffs of a knot and then (untwisted) Whitehead double all components, with various signed clasps. We define the \textit{banded ramified Whitehead doubling} satellite operators by adding bands to a ramified Whitehead double pattern: if there are $n$ total parallel pushoffs coming from the ramification, we may add any choice of up to $n-1$ bands between different components in the pattern. Examples are shown in Figure~\ref{fig:link}.
 \begin{figure}[ht!]
    \centering
        \begin{subfigure}{0.27\textwidth}
            \begin{overpic}[width=\textwidth, 
             unit=1mm, tics=5]{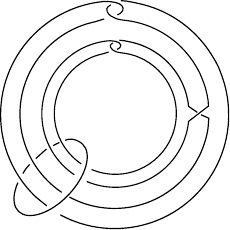}
            \end{overpic}
\end{subfigure}\hspace{15mm}
        \begin{subfigure}{0.4\textwidth}
            \begin{overpic}[width=\textwidth, 
             unit=1mm, tics=5]{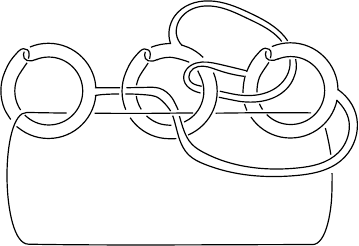}
            \end{overpic}
        \end{subfigure}
\caption{Two ways to banded-ramified-Whitehead double one component of a Hopf link.}
\label{fig:link}
\end{figure}
\begin{theorem}\label{thm:links}Any link obtained by performing $n$ rounds of 
    banded ramified Whitehead doubling operations
    to one component of a Hopf link, for $n\ge 4$, is topologically slice.
\end{theorem}

To our knowledge, previous methods and results are not able to slice every member of this family, owing to the general choice of bands added at each stage.

\subsection{Organization}
In Section~\ref{sec:embeddings} we define generalized Casson handles, discuss relevant properties, and prove Theorem~\ref{thm:open2handles} and Corollary~\ref{cor:skyscrapers}. 
In Section~\ref{sec:handle} we give examples of generalized Casson handles, including $(TH,\partial_-TH)$ and $(BH_m, \partial_- BH_m)$, and define the families of \textit{triangular} and \textit{banded} Casson handles. We also construct the manifolds of Theorem~\ref{thm:mobius}.  In Section~\ref{sec:stein} we Stein realize $\cR$ and the branched covers of Theorem~\ref{thm:mobius}. Using the adjunction inequality, we compute the slice genus of the attaching circle of $TH$ and prove Theorem~\ref{thm:mobius}.
In Section~\ref{sec:exotic} we give Heegaard Floer background, a modification of Gadgil's symplectic nonvanishing theorem for end Floer homology, and prove Theorem~\ref{theorem:exoticR4}. In Section~\ref{sec:disk} we show how to obtain a four-stage core disk in any banded Casson handle, then we prove Theorems~\ref{thm:TH} and \ref{thm:links}.

\subsection{Acknowledgements}  I wish to especially thank my advisor John Etnyre, for many helpful discussions and for his continuous guidance and support, which have been invaluable. I thank Bob Gompf for helpful discussions including background on triangular families, and for insight into end-sums. I thank Yikai Teng for helpful discussions, and for pointing out concerns with an earlier version of Theorem~\ref{thm:open2handles}. I also thank Jen Hom, Tye Lidman, Arunima Ray, Mark Powell, Shelly Harvey, Ryan Dickmann, Gheehyun Nahm, Jacob Guynee, Alex Nolte, and Jay Patwardhan for helpful discussions. The author was partially supported by NSF grant DMS-2203312.
\section{Generalized Casson handles}\label{sec:embeddings}

We begin by defining \textit{Generalized Casson Handles}; see \cite{chapowell, elihomlidman, gompfgenera} for detailed background on Casson handles. These handles include Casson handles, $TH$ and triangular Casson handles (defined in Section \ref{sec:handle}), all infinite towers in the sense of disk embedding theory \cite{diskembedding}, and the class of generalized Casson handles considered by Gompf \cite{gompfsteinisotopy} obtained by stacking self-plumbed disk bundles over surfaces. The freedom allowed in this construction will allow us to construct generalized Casson handles with several good properties, in Section~\ref{sec:handle}.

Recall that Freedman proved the topological Poincar\'e conjecture by showing that Casson handles are homeomorphic to open 2-handles \cite{freedman}. Previously,
it was known that Casson handles had the correct \textit{proper homotopy type}. Moreover, Casson's and Siebenmann's work \cite{cassonthree} showing that Casson handles are proper homotopy equivalent to open 2-handles only relies on certain abstract properties. We use these abstract properties to define the class of {Generalized Casson handles}. Heuristically, generalized Casson handles are obtained by stacking certain pieces which become standard 2-handles, after 2-handles are attached to special loops. Also, each piece is required to give a nullhomotopy of its own attaching region.

\begin{definition}[Generalized Casson Handles]\label{def:gch}
    Suppose for each $k\ge 1$ we have a compact cobordism $h_k = (H_k; \partial_- H_k,\emptyset)$ which is a disjoint union of 4-dimensional 2-handles, and a compact cobordism $n_k = (N_k; \partial_- N_k, \partial_+ N_k)$ with canonical identifications $\partial_- N_k= \partial_- H_k$ and $\partial_+ N_k = \partial_- H_{k+1}$ (hence also $\partial_- N_{k+1} = \partial_+ N_k$) such that the following hold:
\begin{enumerate}
    \item For each $k$ there exists an isomorphism of cobordisms $\theta_k : n_kh_{k+1} \xrightarrow[]{\cong} h_k$, that is, when $N_k$ and $H_{k+1}$ are glued using the canonical identification $\partial_+ N_k = \partial_- H_{k+1}$, the resulting pair $(N_k \cup H_{k+1},\partial_-N_k)$ is diffeomorphic rel boundary to $(H_k, \partial_- H_k)$ via $\theta_k$.
    \item For each $k$, $\partial_-N_k\hookrightarrow N_k$ is componentwise nullhomotopic.
    \item $\partial_- N_1$ is connected (and thus $H_1 \cong D^2\times D^2$.)
\end{enumerate}

See Figure~\ref{fig:GCH}. Write $p_k = (P_k;\partial_- P_k, \partial_+ P_k)$ for the composite cobordism $n_1n_2\hdots n_k$ made from the canonical identifications.
Given $n_k, h_k, p_k$ as above, if $V$ is the interior of the direct limit $\bigcup_{k=1}^\infty P_k$, union the interior of the lower boundary $\partial_- P_1 = \partial_-N_1 = \partial_- H_1$, we will call $(V, \partial V)$ a \textit{Generalized Casson Handle.}
\end{definition}
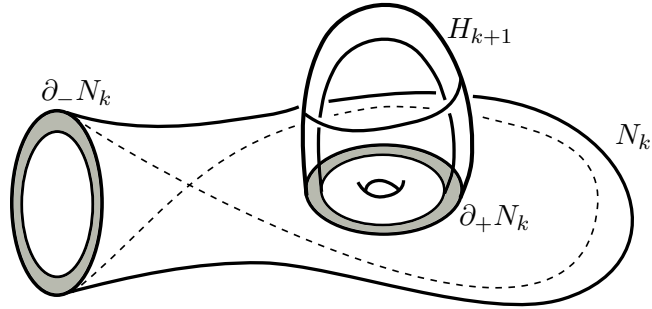
\begin{figure}[ht!]
\centering
            \begin{overpic}[width=0.5\textwidth, 
 unit=1mm, tics=5]{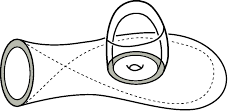}
         \put(70, 43){$H_{k+1}$}
         \put(72, 13){$\partial_+ N_{k}$}
         \put(5, 33){$\partial_- N_{k}$}
         \put(97, 26){$ N_{k}$}
            \end{overpic}
            \caption{Schematic for the $N_k$ and $H_{k+1}$ of a generalized Casson handle.}
\label{fig:GCH}
\end{figure}

 Many examples are given in Section~\ref{sec:handle}. Note that the canonical identifications provide framings for $\partial_-N_k$ and $\partial_+ N_k$, and in general, the $N_k$ will have more boundary than $\partial_- N_k\cup\partial_+ N_k$. In order to prove Theorem~\ref{thm:open2handles}, we state several consequences of the abstract properties above, all of which were proven in Appendix~A to Casson's Lecture~1 \cite{cassonthree}. 

\begin{theorem}[Casson, Lecture 1, Appendix A, Theorem 1 \cite{cassonthree}]\label{thm:abstractcasson}
     Let $(V,\partial V)$ be a generalized Casson handle.
     \begin{enumerate}
         \item There is a canonical embedding $\psi: (V,\partial V) \hookrightarrow (H_1,\partial_-H_1)$ into the standard 2-handle, extending the identity $\partial_- N_1 = \partial_- H_1$. {The embedding has open image. }
         \item  $(V,\partial V)$ is proper homotopy equivalent to $(D^2\times \R^2, S^1\times \R^2)$.
         \item \label{it:inclusion} Writing $S^1 = J\cup J'$, the union of two closed intervals intersecting at their endpoints, the inclusion of $J\times \R^2\subset S^1\times \R^2 = \partial V$ into $V$ is a proper homotopy equivalence.
     \end{enumerate}
\end{theorem}

We also need the following Lemma~\ref{lem:homotopy}, which was essentially proven in Appendix~A to Casson's Lecture~1~\cite{cassonthree}. For clarity, we will include this proof. First we recall some facts and notation, following \cite{cassonthree}. Suppose $(V,\partial V)$ is a generalized Casson handle, with $H_k, N_k,$ and $P_k$ as above. Casson's proof of the embedding $\psi$ above proceeds by defining $\overline{\theta_k}: n_1n_2...n_kh_{k+1}\xrightarrow{\cong} n_1...n_{k-1}h_k$ as an extension of $\theta_k : n_kh_{k+1}\to h_k$ by the identity. Then, letting $\psi_k = \overline{\theta_1}...\overline{\theta_k}: n_1...n_kh_{k+1} \xrightarrow{\cong} h_1$, and defining $\psi|_{n_1...n_k}$ to be $\psi_k|_{n_1...n_k}$, we obtain a well defined map $\psi$ on all of $V$. Figure~\ref{fig:chdiagram} illustrates the image of $\psi_3$ in the 2-handle $H_1$. The embedding $\psi$ allows us to assume $n_kh_{k+1}$ is exactly equal to $h_k$, and that $P_1\subset P_2\subset ... \subset H_1$, and that $H_1\supset H_2\supset ... $ inside $H_1$. Moreover, letting $\partial_+H_k =  \overline{\partial H_k\setminus \partial_-H_k}$, we have $\partial_+H_k\subset \partial_+H_1$ for all $k$. We emphasize that $\partial_+H_k$ and the $\partial_+ N_k$ of Definition~\ref{def:gch} differ, in that $\partial_+N_k$ is only a subset of $\overline{\partial N_k\setminus\partial_-N_k}$.

It follows that the compact set $Y = H_1\setminus V$ may be written as $X\cup \partial_+H_1$, where $X = \bigcap_{k\ge 1} H_k$. From this, Casson describes a system of neighborhoods of the end of $V$: let $\{\partial_+^i\}$ be a fundamental system of (closed) collared neighborhoods of $\partial_+H_1$ in $H_1$, such that $\text{int}(\partial_+^i)\supset \partial_+^{i+1}$ for all $i$, and such that $\partial_+^i \cap H_i$ is a collar of $\partial_+H_i$ for each $i$ (as in \cite[p. 217]{cassonthree}). Then a fundamental system of (closed) neighborhoods of $Y$ in $H_1$ is $\{H_i\cup \partial_+^i\}$ (hence, a system of closed (in $V$) neighborhoods of the end of $V$ is $\{(H_i\cup \partial_+^i)\setminus Y\}$). See Figure~\ref{fig:chdiagram}.

\begin{figure}[ht!]
\centering
            \begin{overpic}[width=0.5\textwidth, 
 unit=1mm, tics=5]{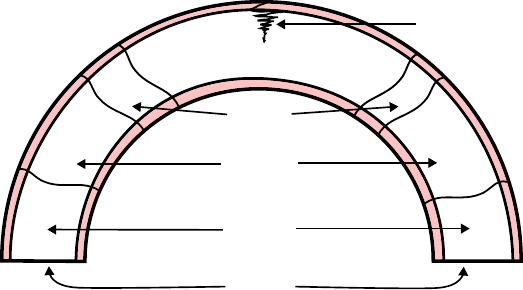}
 \put(80, 49){$X$}
         \put(62, 43){$H_4$}
         \put(47, 32){$N_3$}
         \put(47, 22){$N_2$}
         \put(47,10){$N_1$}
         \put(45, 0){$\partial_-H_1$}
            \end{overpic}
            \caption{The 2-handle $H_1$, with embedded $H_2$, $H_3$, and $H_4$ visible (e.g. $H_3 = H_4\cup N_3$), and a collar of $\partial_+H_1$ (shaded). Ignoring the set $X$, we see the image of $\psi_3$ in $H_1$.}
\label{fig:chdiagram}
\end{figure}
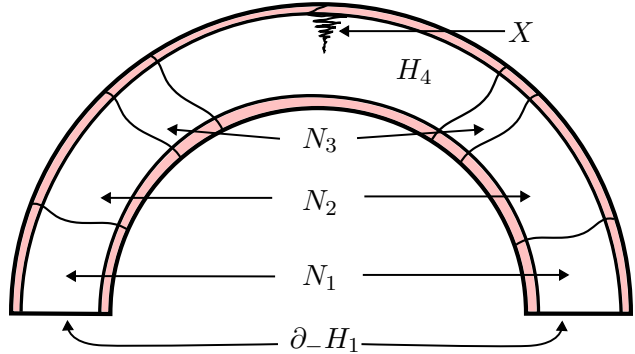
The proof of Lemma~\ref{lem:homotopy} below is the same as Casson's \cite[Proof of Step 3., Appendix A]{cassonthree}.
\begin{lemma}\label{lem:homotopy}
    For each $n\ge 1$ there is a (not necessarily proper) homotopy $g_t : H_{n+1}\cup \partial_+^{n+1}\to H_{n}\cup \partial_+^{n}$, where $g_0$ is the inclusion, $g_t$ fixes $\partial_+^{n+1}$ pointwise, and $g_1$ has image in $\partial_+^n$.
\end{lemma}
\begin{proof}
    It suffices to find a homotopy of the inclusion $H_{n+1}\hookrightarrow H_n$, rel $H_{n+1}\cap\partial_+^{n+1}$, to a map into $H_n\cap\partial_+^n$. Then the paste of this homotopy with the identity on $\partial_+^{n+1}$ is the desired homotopy $g_t$. Note, any 2-handle deformation retracts, rel a collar of its upper boundary, to its cocore union the collar. Thus, since each $\partial_+^i$ restricts to a collar of $\partial_+H_i$, the existence of the homotopy above is equivalent to the inclusion map $\pi_2(H_{n+1},\partial_+H_{n+1})\to \pi_2(H_n,\partial_+H_n)$ being zero. By Hurewicz, this is the same as $H_2(H_{n+1},\partial_+H_{n+1})\to H_2(H_n,\partial_+H_n)$ being zero, which amounts to showing the cocore of each 2-handle of $H_{n+1}$ has intersection number zero with the core of each 2-handle of $H_n$. These numbers are the same as the linking numbers in $\partial H_n \cong S^3$ of the cores of the solid tori in $\partial_+H_{n+1}$ and $\partial_- H_n$, which are zero since $\partial_-H_n = \partial_- N_n$ are componentwise nullhomotopic in $N_n = H_n\setminus \mathrm{int}(H_{n+1})$ by Definition~\ref{def:gch}, Item 2.  
\end{proof}
\begin{remark}
    Lemma~\ref{lem:homotopy} still holds if each space is replaced with its intersection with $(\partial_+H_1)^c$ (ie by restricting $g_t$). By an abuse of notation, we assume this is done throughout the following.
\end{remark}
We now show that Theorem~\ref{thm:abstractcasson}, Lemma~\ref{lem:homotopy}, and a proper $h$-cobordism theorem combine to give Theorem~\ref{thm:open2handles}, which states all generalized Casson handles are topologically standard.

\begin{remark}
    Note that the logic is not circular, but iterative. In particular, Freedman and Quinn directly showed certain infinite towers are topologically standard \cite{freedman, freedmanquinn}. These were then used to prove the proper $h$-cobordism theorem, which in turn we use to prove generalized Casson handles are topologically standard. Thus, we are bootstrapping from the specific infinite towers used by Freedman and Quinn to the more general, generalized Casson handles.
\end{remark} 
In \cite{gompfsteinisotopy}, Gompf standardized another class of generalized Casson handles directly, similarly to Freedman's proofs that Casson handles are standard. It is clear that class (also Casson handles, Triangular Casson handles, etc.) is a subclass of ours but not known whether it is proper. 
As in the introduction, we remark that Theorem \ref{thm:open2handles} shows that all \textit{open infinite towers} of disk embedding theory \cite[Definition 12.16]{diskembedding} are topologically standard: thus, while certain growth conditions (\textit{boundary-shrinkable} and \textit{replicable} \cite{diskembedding}) are needed to force the endpoint compactification of an infinite tower to be homeomorphic to a closed 2-handle, these are not necessary if one just wants a topologically flat disk bounded by the attaching circle. 

The proper $h$-cobordism theorem of Freedman-Quinn, simplified to the case of one end with fundamental group $\bbZ$ at infinity, is as follows. Recall that the fundamental group at infinity of an end is said to be \textit{stable} if given a sequence of neighborhoods defining the end, their inclusions eventually induce isomorphisms of $\pi_1$.

\begin{theorem}[\cite{freedmanquinn}, Corollary 7.3B]\label{thm:properhcobordism}
    Let $(W;V,V')$ be a simply connected smooth proper $h$-cobordism of dimension 5, with one end. Suppose $W$ (and therefore $V$ and $V'$) have stable fundamental group $\bbZ$ at infinity, and if $C = \partial W - (\mathrm{int} \,V\cup \mathrm{int}\, V')$ is not empty, assume that $C$ is diffeomorphic to $(C\cap V) \times I$. Then $W$ is homeomorphic to $V\times I$ extending the product structure on $C$.
\end{theorem}

We state a proper homotopy equivalence criterion due to Casson-Siebenmann \cite{cassonthree}. 
Throughout we assume all spaces are locally finite Hausdorff simplicial complexes.
Recall a map of spaces is \textit{proper} if preimages of compact sets are compact, and a \textit{proper homotopy} is a homotopy which itself is a proper map. A \textit{proper homotopy equivalence} $f:X\to Y$ is a proper map, such that there exists a proper map $g: Y\to X$ with $g\circ f$ and $f\circ g$ homotopic to the respective identities via proper homotopies. Concatenation of proper homotopies yields proper homotopies, and proper homotopy equivalence is an equivalence relation. The following criterion for checking proper homotopy equivalence is proven in Casson's Lecture 1, Appendix B.

\begin{theorem}[Casson-Siebenmann, Appendix B, \cite{cassonthree}]\label{thm:cassonproper}
    Suppose $X\subset Y$ is a closed subspace and a strong deformation retract. Suppose for each compact $K\subset Y$ there is a compact $L\subset Y$ containing $K$, and a (not necessarily proper) homotopy $g_t:Y\setminus L\to Y\setminus K$ fixing $X$ pointwise, with $g_0$ the inclusion, and $g_1$ having image in $X$. Then $X$ is a proper deformation retract of $Y$.
\end{theorem}
\begin{proof}[Proof of Theorem \ref{thm:open2handles}]
     We construct a simply connected, single ended proper $h$-cobordism $W$ from
     $(V,\partial V)$
     to $(D^2\times \bbR^2, S^1\times \bbR^2)$ rel boundary, with stable fundamental group $\bbZ$ at infinity. The proof then follows by Theorem~\ref{thm:properhcobordism}. In Step 1 we construct the cobordism $W$. In Step 2 we use Theorem~\ref{thm:cassonproper} to show the inclusion of the standard 2-handle is a proper homotopy equivalence: to do this, we introduce a compact exhaustion of $W$, whose complements we can homotop into the standard 2-handle using Lemma~\ref{lem:homotopy}. In Step 3 we show the inclusion of the generalized Casson handle into $W$ is a proper homotopy equivalence, concluding the proof.

     \textbf{Step 1.} By Theorem \ref{thm:abstractcasson} there exists a smooth embedding $\psi: (V,\partial V)\hookrightarrow(D^2\times \bbR^2, S^1\times \bbR^2)$ extending the identity map on the boundary. 
     Call the open 2-handle $h$. Using the embedding, construct the following manifold $W$, in the same spirit as \cite[Corollary 1.2]{freedman}: 
     \[W = (h\times [0,1)) \cup (V\times \{1\}).\]
      See Figure~\ref{fig:phc}, left. This is a manifold since $\psi$ has open image. We will show $W$ is a proper $h$-cobordism. It is clear that the inclusions of $h\times  \{0\}$ and $V\times\{1\}$ are proper maps, inducing homeomorphisms of end spaces, all of which are a single point; moreover both boundary components and $W$ are contractible. Thus $W$ is a simply connected $h$-cobordism, and the subset of the boundary $\partial W \setminus ( \operatorname{int}((h\times\{0\})\cup (V\times \{1\}))$ is diffeomorphic to a product $(S^1\times \R^2)\times I$.

         \begin{figure}[ht]
\centering
\begin{subfigure}[b]{0.33\textwidth}
   \includegraphics[scale=0.8]{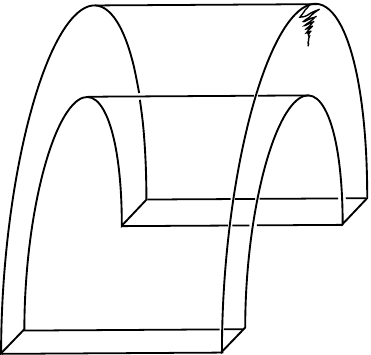} 
\end{subfigure}\begin{subfigure}[b]{0.33\textwidth}
    \includegraphics[scale=0.8]{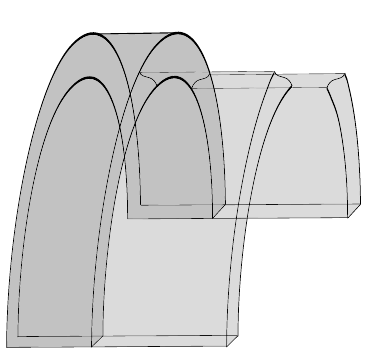} 
\end{subfigure}\begin{subfigure}[b]{0.33\textwidth}
   \includegraphics[scale=0.8]{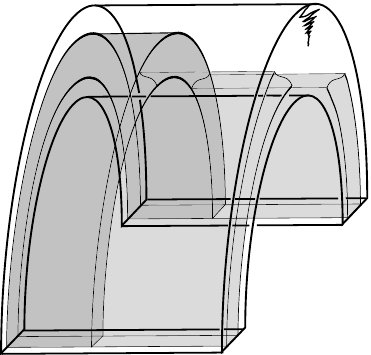} 
\end{subfigure}
\caption{(Left) Schematic of the proper $h$-cobordism $W$. The left face is $h\times \{0\}$; the right face is $V\times \{1\}$, which is obtained by deleting the set $X$ (the squiggle) from $h\times \{1\}$. (Middle) Compact exhaustant $\cK_n$ of Step 2 (with $\ve_n$ rescaled for clarity). (Right) $\cK_n$ inside $W$.}
\label{fig:phc}
\end{figure}

\textbf{Step 2.} We show the inclusion $i:h \times \{0\} \hookrightarrow W$ is a proper homotopy equivalence using Theorem~\ref{thm:cassonproper}. First define a compact exhaustion of $W$ as follows. Note, the interiors of the collars $\{\partial_+^i\}$ of Lemma~\ref{lem:homotopy} are the complements of compact exhaustants $\{K_n\}$ of the open 2-handle $h$ (which are of the form $D^2\times nD^2$). Also $V$ admits a compact exhaustion $\{V_n\}$, where $V_n$ is the closure of $V\setminus ((H_n \cup \partial_+^n) \setminus X)$ in $V$. Moreover, $V_n\subset K_n$ inside the 2-handle. Then choosing $0 < \ve_n \searrow 0$, the sets
\[\cK_n:= (K_n\times [0,1-\ve_n]) \cup (V_n\times [1-\ve_n,1]) \] form a compact exhaustion of $W$ with $\cK_n\subset \text{int}(\cK_{n+1})$. Note, for each $n$ we have
\[W\setminus \cK_n = \left(\text{int}(\partial_+^n) \times [0,1-\ve_n]\right) \cup \left( (h\setminus V_n)\times(1-\ve_n,1) \right) \cup \left((V\setminus V_n) \times\{1\}\right).\] Let $K\subset W$ be compact and choose $K\subset \text{int}(\cK_m)$. Then $K$ is disjoint from the closure of $W\setminus \cK_m$, which is 
\begin{align*}
\overline{W\setminus \cK_m} &= \left(\partial_+^m \times [0,1-\ve_m]\right) \cup \left(\overline{(h\setminus V_m)}\times[1-\ve_m,1)\right) \cup \left(\overline{(V\setminus V_m)} \times\{1\}\right)\\
&= \left(\partial_+^m \times [0,1-\ve_m]\right) \cup \left((H_m \cup \partial_+^m)\times[1-\ve_m,1)\right) \cup \left((H_m \cup \partial_+^m)\setminus X) \times\{1\}\right).\end{align*}
Moreover, since $K\subset \cK_m \subset \cK_{m+1}$, we have $W\setminus \cK_{m+1}\subset\overline{W\setminus \cK_{m+1}} \subset \overline {W\setminus \cK_m} \subset W\setminus K$. By Theorem~\ref{thm:cassonproper}, it suffices to show there is an ordinary homotopy $g_t : \overline{W\setminus \cK_{m+1}} \to \overline {W\setminus \cK_m}$ starting with the inclusion map, fixing $h\times \{0\}$ pointwise, and such that $g_1$ has image in $h\times\{0\}$.

Using a straight-line homotopy near $\{1\}$, $\overline{W\setminus \cK_{m+1}}$ deformation retracts inside $\overline{W\setminus \cK_{m}}$ to
\[ \left(\partial_+^{m+1} \times [0,1-\ve_{m+1}]\right) \cup \left((H_{m+1} \cup \partial_+^{m+1})\times[1-\ve_{m+1},x]\right) \] for some $1-\ve_{m+1} < x < 1.$ See Figure~\ref{fig:phc}. The two sets in the union above are closed, and intersect in the closed set $\partial_+^{m+1}\times \{1-\ve_{m+1}\}$. By Lemma~\ref{lem:homotopy} there exists a homotopy $g_t' : H_{m+1}\cup\partial_+^{m+1} \to H_{m}\cup\partial_+^{m}$, which fixes $\partial_+^{m+1}$ pointwise, to a map with image in $\partial_+^m$. By applying $g_t'$ fiberwise to the second set above, and extending by id to the first set, we obtain a homotopy $g_t'' : \overline{W\setminus \cK_{m+1}} \to \overline {W\setminus \cK_m}$ starting with the inclusion map, fixing $h\times \{0\}$ pointwise, such that $g_1''$ has image contained in
\[ \left(\partial_+^{m+1} \times [0,1-\ve_{m+1}]\right) \cup \left(\partial_+^m\times[1-\ve_{m+1},x]\right).\]
See Figure~\ref{fig:homotopy}. Thus, $g_1''$ has image contained in $\partial_+^m\times [0,x]$, and this set is disjoint from $K$, since we showed above \[\partial_+^m\times [0,x]\subset \overline {W\setminus \cK_m} \subset W\setminus K.\] Using a straight-line homotopy to push $\partial_+^m\times [0,x]$ into $h\times \{0\}$, we obtain the desired $g_t$.

\begin{figure}[h!]
        \centering 
         \begin{subfigure}[b]{0.46\textwidth}
            \begin{overpic}[width=\textwidth, 
 unit=1mm, tics=5]{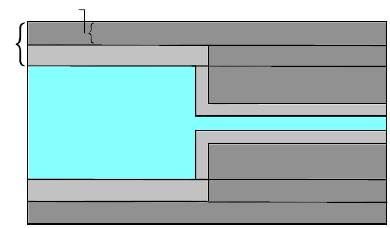}
 \put(25, 25){$\cK_m$} 
 \put(0, 55){$\partial_+^{m+1}\times I$} 
  \put(-13, 45){$\partial_+^{m}\times I$}
 \end{overpic}
        \end{subfigure} \hspace{5mm}\begin{subfigure}[t]{0.46\textwidth}
            \begin{overpic}[width=\textwidth, 
 unit=1mm, tics=5]{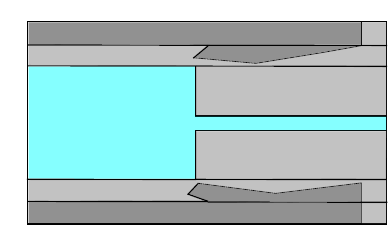}
 \end{overpic}
        \end{subfigure}
       \caption{(Left) Diagram of $W$ showing $\overline {W\setminus \cK_m}$ in dark union light grey, $\overline {W\setminus \cK_{m+1}}$ in dark grey, $h\times \{0\}$ as the left edge, and $V\times \{1\}$ as the right edge. Here the sets $\partial_+^n\times I$ are implicitly intersected with $V\times \{1\}$ (ie they have the set $X\times \{1\}$ deleted.) (Right) The homotopy $g_t''$ (constructed using Lemma~\ref{lem:homotopy}) squeezes the dark grey, through the light grey, into $\partial_+^m\times[0,x]$, for $x < 1$.}
        \label{fig:homotopy}
    \end{figure}

We now see $W$ has the proper homotopy type of the open 2-handle, and thus has stable fundamental group $\bbZ$ at infinity (since this is preserved under proper homotopy equivalence \cite[p. 490]{siebenmann}.) Note, the stable $\pi_1^\infty = \bbZ$ and single-endedness can also be seen using the system $\{W\setminus \cK_n\}$.

\textbf{Step 3.} We show the inclusion $j: V\times\{1\} \to W$ is a proper homotopy equivalence. By Theorem \ref{thm:abstractcasson}, Item \ref{it:inclusion}, if $J\subset S^1$ is a closed interval, the inclusion of $J\times \R^2 \subset S^1\times \R^2 = \partial V$ into $V$ is a proper homotopy equivalence. By construction, $W$ has a product structure restricted to $\partial W \setminus \text{int}(h\cup V) \cong (S^1\times \R^2)\times [0,1]$; fix an identification. Then the embedding $J\hookrightarrow S^1$ gives us an embedding \[(J\times \R^2)\times [0,1] \hookrightarrow (S^1\times \R^2)\times [0,1] \cong \partial W \setminus \text{int}(h\cup V) .\] Thus, we obtain a commutative diagram of proper maps, where all arrows are inclusions: 
\[\begin{tikzcd}
	{h\times\{0\}} && {W} && V\times \{1\}\\
	\\
	{(J\times \R^2)\times\{0\}} && {(J\times \R^2)\times[0,1]} && (J\times \R^2)\times \{1\}
	\arrow["{\simeq_p}", from=1-1, to=1-3]
	\arrow["{\simeq_p}", 
    from=3-1, to=1-1]
	\arrow["{\simeq_p}", from=3-1, to=3-3]
	\arrow[from=3-3, to=1-3]
    \arrow["j"',from=1-5, to=1-3]
    \arrow["{\simeq_p}"', from=3-5, to=3-3]
    \arrow["{\simeq_p}"', 
    from=3-5, to=1-5]
\end{tikzcd}\]
It follows that the middle vertical arrow, and hence $j$, are proper homotopy equivalences.
\end{proof}

\begin{proof}[Proof of Corollary \ref{cor:skyscrapers}]
    This follows since open infinite towers are generalized Casson handles.
\end{proof}

\section{Defining some 4-manifolds} \label{sec:handle}
In this section we define relevant 4-manifolds used throughout the paper, starting with ribbon disk complements $X_n$. We then define $(TH,\partial_- TH)$, the general class of Triangular Casson handles, and the smaller class of banded Casson handles. Next we give an infinite family $BH_m$ of generalized Casson handles that are hybrids of infinite towers and triangular Casson handles. We use the $BH_m$ to construct open 4-manifolds $\A_m$ and $\M_m$ that will serve as branched covers over irreducible surfaces in $\R^4$. 
\begin{figure}[h!]
        \centering 
         \begin{subfigure}[b]{0.2\textwidth}
            \begin{overpic}[width=\textwidth, 
 unit=1mm, tics=5]{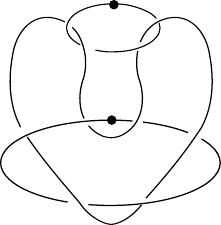}
 \put(93, 85){$n$} 
 \end{overpic}
        \end{subfigure} \hspace{10mm}\begin{subfigure}[t]{0.3\textwidth}
            \begin{overpic}[width=\textwidth, 
 unit=1mm, tics=5]{figures/diskcomplement8.pdf}
 \put(42, 35){$n-2$}
 \end{overpic}
        \end{subfigure}
       \caption{Ribbon disk complement $X_n$ and the ribbon knot $K_n$.}
        \label{fig:diskcomplement3}
    \end{figure}
\subsection{The ribbon disk complements $X_n$}

By generalizing the $-1$-framing in Figure~\ref{fig:exoticR4} we obtain the ribbon disk complement $X_n$ shown in Figure~\ref{fig:diskcomplement3}. 
    \begin{figure}[htb!]
        \centering 
 \begin{subfigure}[b]{0.23\textwidth}
            \begin{overpic}[width=\textwidth, 
 unit=1mm, tics=5]{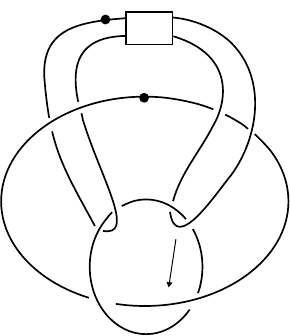}
 \put(63, 21){$n$} 
 \put(39, 89){$-1$}
 \end{overpic}
            \caption{}
        \end{subfigure} \hspace{5mm}
 \begin{subfigure}[b]{0.23\textwidth}
            \begin{overpic}[width=\textwidth, 
 unit=1mm, tics=5]{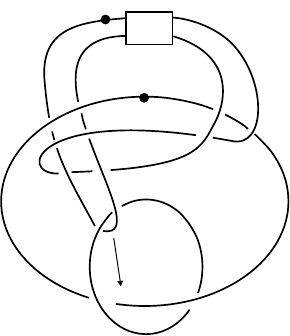}
 \put(63, 21){$n$} 
 \put(39, 89){$-1$}
 \end{overpic}
            \caption{}
        \end{subfigure} 
        \hspace{5mm}
 \begin{subfigure}[b]{0.25\textwidth}
            \begin{overpic}[width=\textwidth, 
 unit=1mm, tics=5]{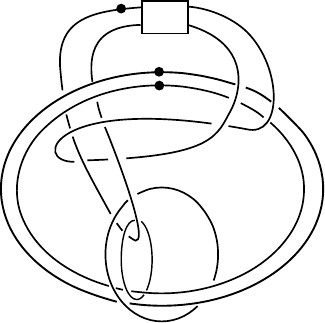}
 \put(70, 21){$n$} 
 \put(48, 21){0}
 \put(45, 91){$-1$}
 \end{overpic}
            \caption{}
\end{subfigure}\hspace{10mm}\begin{subfigure}{0.23\textwidth}
            \begin{overpic}[width=\textwidth, 
 unit=1mm, tics=5]{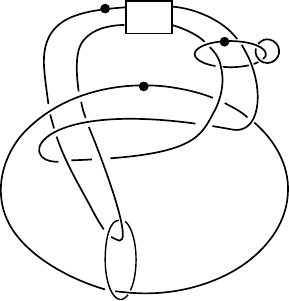}
 \put(46, 23){0} 
 \put(43, 92){$-1$}
  \put(92, 87){$n$}
 \end{overpic}
            \caption{}
            \label{fig:generaldiskcomplement}
        \end{subfigure}\hspace{10mm}
 \begin{subfigure}[b]{0.23\textwidth}
            \begin{overpic}[width=\textwidth, 
 unit=1mm, tics=5]{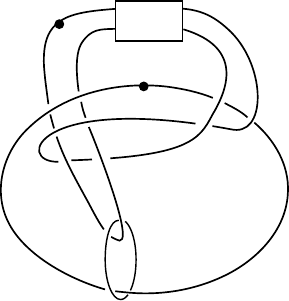}
 \put(46, 23){0} 
 \put(38.5, 91){$n-1$}
 \end{overpic}
            \caption{}
        \end{subfigure} \hspace{10mm}
 \begin{subfigure}[b]{0.3\textwidth}
            \begin{overpic}[width=\textwidth, 
 unit=1mm, tics=5]{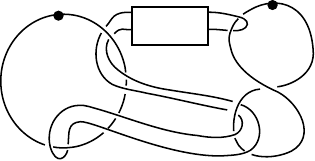}
             \put(54, 24){0} 
 \put(45, 40){$n-1$}
 \end{overpic}
            \caption{}
        \end{subfigure} \hspace{10mm}
 \begin{subfigure}[b]{0.32\textwidth}
            \begin{overpic}[width=\textwidth, 
 unit=1mm, tics=5]{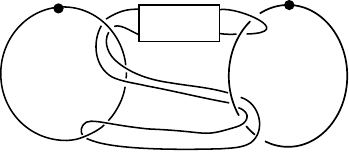}
             \put(50, 22){0} 
 \put(43, 36){$n-2$}
 \end{overpic}
            \caption{}
        \end{subfigure} 
       \caption{Kirby calculus showing $X_n$ is a disk complement for $K_n$. (a) and (b) indicate 1-1 slides. (e) is the result of sliding the $n$ framed 2-handle over the 0-framed 2-handle, bringing the resulting cancelling pair around then sliding the 0-framed 2-handle over the $n$-framed one, twisting the band.}
        \label{fig:diskcomplement}
    \end{figure}
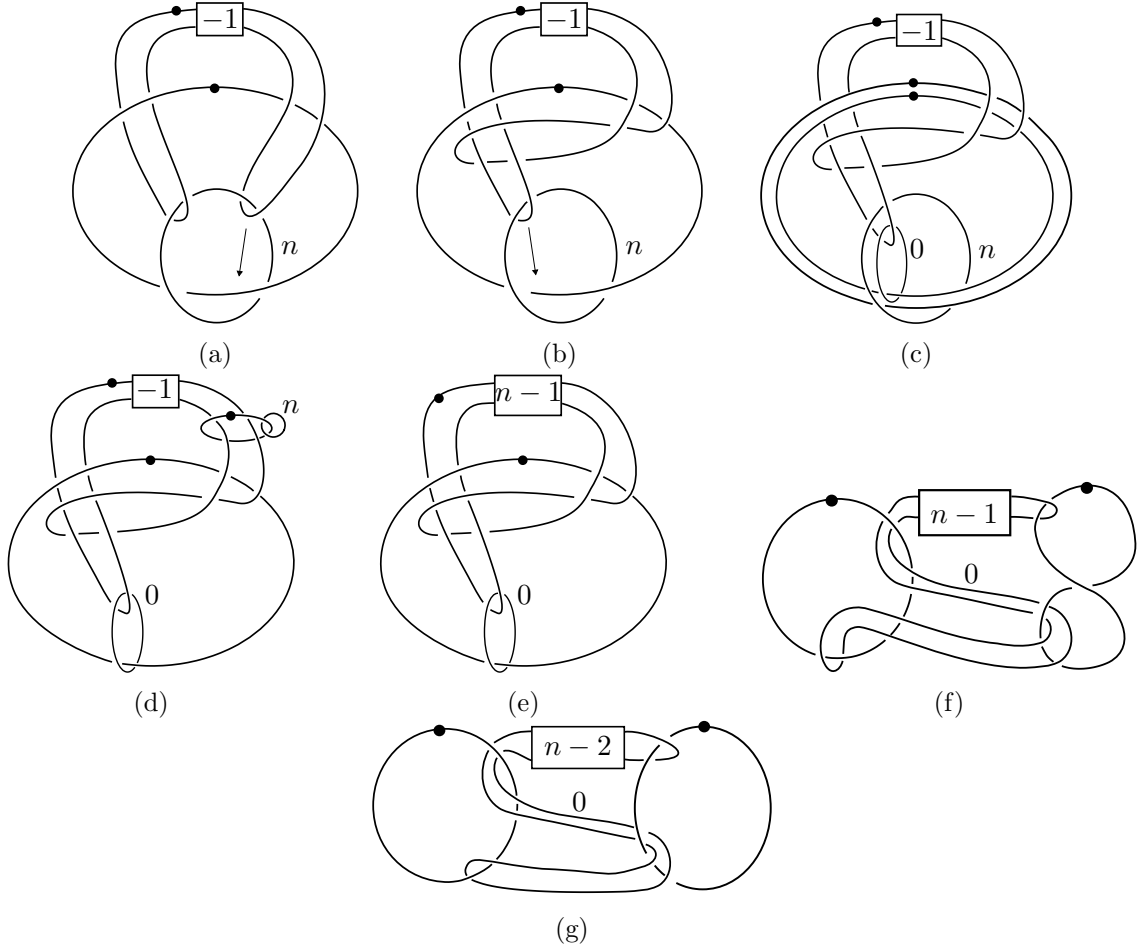
\begin{proposition}\label{prop:diskcomplements}
    Let $n\in \bbZ$. The manifolds $\{X_n\}$ of Figure~\ref{fig:diskcomplement3} are ribbon disk complements for the knots $K_n$ of Figure~\ref{fig:diskcomplement3}, which are all genus 2.
\end{proposition}

\begin{proof}
    Each $X_n$ is a ribbon disk complement since it is a dotted circle in a Kirby diagram for $B^4$. 
    Kirby moves describing the knots from the disk complement are shown in Figure \ref{fig:diskcomplement}. The resulting knots $K_n$ are twisted band sums of a 2-component unlink; work of Hedden-Watson \cite[Theorem 1.3]{heddenwatson} or Wang \cite[Theorem 1]{wangcrossing} shows the resulting knots have the same knot Floer homology, hence the same genus. When $n = 0$ we obtain $10_{140}$ which is genus 2 \cite{knotinfo}. 
\end{proof}
We mainly use $X_{-1}$ in this paper; we comment on the $X_n$'s at the end of Section~\ref{sec:detecting}. Proposition \ref{prop:diskcomplements} is stated as such for convenience. For this reason we also note $n=-1$ corresponds to 12n582 (from Knotinfo \cite{knotinfo}), and all $K_n$ are fibered.

\subsection{$TH$ and triangular Casson handles}\label{sec:THTriangular}
 \begin{figure}[htb!]
\centering
\begin{subfigure}[b]{0.13\textwidth}
            \begin{overpic}[width=\textwidth, 
 unit=1mm, tics=5]{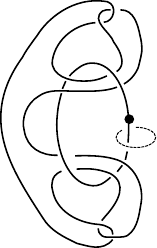}
 \put(65, 40){$a$}
  \put(3, 85){$C$}
 \end{overpic}
            \caption{}
            \label{fig:TH}
        \end{subfigure}\hspace{10mm}
        \begin{subfigure}[b]{0.16\textwidth}
            \begin{overpic}[width=\textwidth, 
 unit=1mm, tics=5]{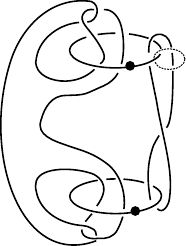}
        \put(71, 40){0} 
         \put(76, 78){$a$}
  \put(-2, 82){$C$}
            \end{overpic}
            \caption{}
            \label{fig:TH2}
        \end{subfigure} \hspace{10mm}
        \begin{subfigure}[b]{0.4\textwidth}
            \begin{overpic}[width=\textwidth, 
 unit=1mm, tics=5]{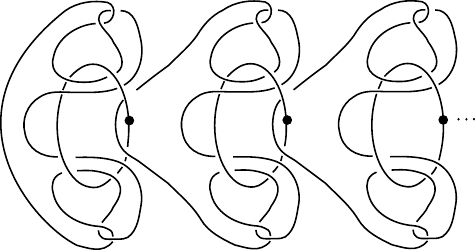}
 \put(71, 45){0} 
 \put(37, 45){0} 
   \put(2, 44){$k$}
 \end{overpic}
            \caption{}
            \label{fig:TH3}
        \end{subfigure}
\caption{In all figures, the curves $C$ and $a$ are labeled curves in the boundary, without 2-handles attached. (a): Diagram for $Tk$, as $S^1\times B^3$ with attaching circle $C$ and tip circle $a$ in the boundary. (b): Diagram showing $Tk$ as a triangular plumbed 2-handle. (c): Diagram for $TH$ attached to $B^4$ along a $k$-framed unknot.}
\end{figure}
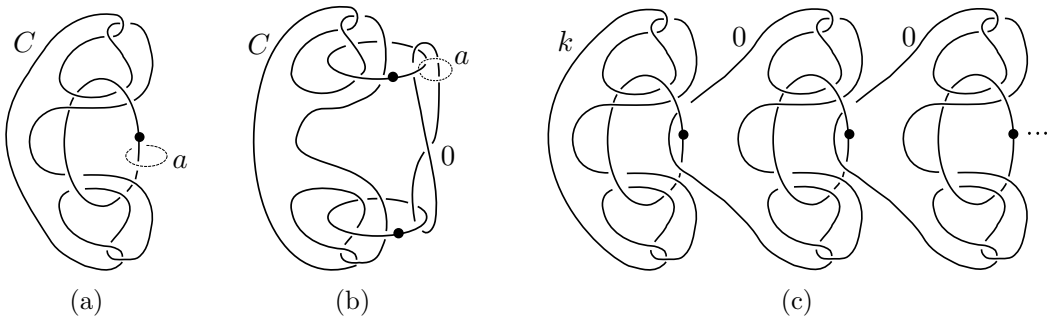
We formally define the manifold $(TH, \partial_-TH)$, making explicit the analogy with a Casson handle.
Like a Casson handle, $TH$ is a noncompact 4-manifold with boundary $\partial_-TH\cong S^1\times \R^2$, with a framed circle in the boundary. We then show that $TH$ is a special case of a \textit{banded Casson handle}, which is a special type of \textit{triangular Casson handle}, which we define later in this section. Banded Casson handles are central to the proof of Theorem~\ref{thm:links} given in Section~\ref{sec:disk}.

We first describe the analogue of the self-plumbed 2-handles $(k,\partial_-k)$, also called kinky handles, used in the construction of Casson handles; these `triangular kinky handles' $(Tk, \partial _-Tk)$ are copies of $S^1\times B^3$ where we have subdivided the boundary in an interesting way. A Kirby diagram of $Tk$ is shown in Figure~\ref{fig:TH}. Precisely, this is a diagram of $S^1\times B^3$ with two circles $C$ and $a$ drawn in the boundary, called the \textit{attaching circle} and \textit{tip circle} of $Tk$, respectively. We define the \textit{attaching region} $\partial_- Tk = \nu C$ to be a closed tubular neighborhood of $C$, and the \textit{upper boundary} to be $\partial_+Tk := \overline{\partial Tk \setminus \partial_-Tk}$. Observe $C$ bounds an immersed disk $D \subset Tk$ with two positive double points, visible in Figure~\ref{fig:TH}; we will call $D$ the \textit{immersed core disk} of $Tk$. Figure~\ref{fig:TH} shows that $D$ induces the blackboard framing on $C$ (we will use this fact in Section~\ref{sec:disk}). The \textit{attaching framing} on $C$, however, is obtained by twisting the blackboard framing in Figure~\ref{fig:TH} by four left twists.
It follows that attaching $Tk$ to a 4-manifold along a framed knot has the same effect on the intersection form on $H_2$ as attaching a standard 2-handle along the same framed knot. To see this, notice Figure \ref{fig:TH} is obtained by scooping out a disk from $B^4$, corresponding to the 1-handle. That disk is disjoint from the curve $C\subset \partial B^4$, which is unknotted. Thus, the pair $(B^4,\nu(C))$ is a standard, unnattached, 2-handle: there is a canonical identification $\nu(C)\cong S^1\times D^2$ under which each $S^1\times pt$ bounds a disjoint disk, comprising the entire $B^4$. With this in mind, the attaching framing we described for $C\subset Tk$ is the same as the product framing on the attaching region $S^1\times D^2\subset D^2\times D^2$ of the standard 2-handle, before we delete the disk corresponding to the 1-handle in Figure~\ref{fig:TH}.  
 The circle $a$ is given the blackboard framing in Figure~\ref{fig:TH}; equivalently, $a$ is framed so that attaching a 0-framed 2-handle to $Tk$ along $a$ yields a standard 2-handle. 

The handle $(TH, \partial_-TH)$ is defined to be an infinite union of repeating $Tk$ units, attached by gluing the $n$th attaching region to the $(n-1)$st tip circle using the framings we described, where we then remove all boundary except the open attaching region of the first stage. The attaching circle $C \subset \partial_-TH$ is framed as in $Tk$. A Kirby diagram of $TH$ attached to $B^4$ along a $k$-framed unknot is given in Figure \ref{fig:TH3}. To interpret the picture, attach the given handles to $\R^3\times [0,1)$, then delete all boundary.

Gompf pointed out that $TH$ can be understood in terms of a construction, initially studied by Freedman in work towards the Poincar\'e conjecture, which generalizes the notion of a Casson handle. These \textit{triangular Casson handles} are infinite towers of self-plumbed 2-handles which have had some of their double point loops homologically paired up by 2-handles. We give a general definition, then describe a specific family.

\begin{definition}[Triangular Family of Curves]\label{def:triangularfamily}
Let $(k,\partial_-k)$ be a self-plumbed 2-handle, and recall $k$ admits a standard Kirby diagram with one 1-handle for each of its double points (See Figure~\ref{fig:triangular}, left, for an example with three double points). Let $\mu_1,...,\mu_n$ be meridians for all of these 1-handles, and let $\partial_+ k = \overline{\partial k\setminus \partial_-k}$. Give the $\mu_i$ their canonical framings (such that attaching 0-framed 2-handles to $\mu_i$ turns $k$ into a 2-handle.) A \textit{triangular family of curves for $k$}, or \textit{triangular family}, is any set of disjoint closed framed curves $\{\gamma_1,...,\gamma_n\}$ in $\partial_+ k$ defined as follows. Let $\gamma_1 = \mu_1$ with the canonical framing. Let $\gamma_2$ be any framed curve that becomes isotopic to $\mu_2$ as a framed curve, after a 0-framed 2-handle is attached to $\gamma_1$. Let $\gamma_3$ be any framed curve that becomes isotopic to $\mu_3$ as a framed curve after 0-framed 2-handles are attached to $\gamma_1$ and $\gamma_2$, and so on until we have defined $\gamma_n$.
\end{definition}

\begin{remark}
    In the above, the $[\mu_i]$ form a basis for $H_1(k)\cong \mathbb{Z}^n$, and so does any triangular family $\{[\gamma_1],...,[\gamma_n]\}$ for $k$. By Definition~\ref{def:triangularfamily}, writing each $[\gamma_j]$ in terms of the $[\mu_i]$ yields a triangular matrix. The triangular families we consider will often consist of several meridians and only a few curves that pair up multiple 1-handles. 
\end{remark}

\begin{remark}
    A similar idea is used by Bi\v zaca-Gompf to obtain the `simplest exotic $\R^4$' from a more complicated one \cite[Proof of Theorem 0.2]{bizacagompf}.
 \end{remark}

We now describe the basic building blocks for \textit{triangular Casson handles}. In this discussion, let $k$ be a self-plumbed 2-handle with meridians $\mu_1,...,\mu_n$ for its 1-handles, and let $\{\gamma_1,...,\gamma_n\}$ be any triangular family for $k$. Choose $0\le m\le n-1$ and attach $m$ 2-handles along $m$ distinct 0-framed curves $\gamma_{i_1},...,\gamma_{i_m}$. The resulting manifold $(k',\partial_- k')$ is a \textit{triangular plumbed 2-handle}. Any such $k'$ has an attaching region $\partial_-k' = \partial_-k$ with framing induced by the plumbed handle $k$. Moreover, $k'$ has $n-m$ \textit{tip circles}, which are the meridians $\mu_j$ such that $j\ne i_1,...,i_m$. The tip circles obtain framings as double point loops for the plumbed handle $k$, and attaching 0-framed 2-handles to them turns $k'$ into a 2-handle, which follows from the definition of triangular families.

\begin{definition}[Triangular Casson towers and handles]
    A \textit{height 1 triangular Casson tower} $(T_1, \partial_-T_1)$ is a triangular plumbed 2-handle. A \textit{height $n$ triangular Casson tower} $(T_n, \partial_-T_n)$ is obtained inductively by attaching triangular plumbed 2-handes to all tip circles of a height $n-1$ triangular Casson tower, using the attaching and tip framings. A \textit{triangular Casson handle} is a union of any infinite sequence $T_1\subset T_2\subset \hdots$ of nested triangular Casson towers constructed this way, minus all boundary except the interior of $\partial_- T_1$.
\end{definition}

 Triangular Casson handles immediately satisfy the three conditions of Definition~\ref{def:gch}, hence by Theorem~\ref{thm:open2handles} are homeomorphic to open 2-handles. 

In Section~\ref{sec:disk} we work with a subclass of triangular Casson towers made using a simplified version of triangular families. These \textit{banded Casson towers} are described as follows. Given a self-plumbed 2-handle $k$, it admits a standard Kirby diagram with dotted circles $c_1,...,c_n$ corresponding to its double points, as in the left diagram of Figure~\ref{fig:triangular}. Using up to $n-1$ disjoint bands in $\partial_+k$, band-sum together various $c_i$. We require that each pair $c_i,c_j$ of dotted circles has at most one band connecting it, and that there are no bands from any circle to itself. Pushing the bands into the interior of $k$ and deleting them amounts to adding 0-framed 2-handles which follow the bands \cite[Section 6.2]{gompfstipsicz}; see Figure~\ref{fig:triangular} for an example. This is a \textit{banded plumbed 2-handle} and it is a special case of a triangular plumbed 2-handle. Similarly one can define \textit{banded Casson towers/handles}. Note, to find the tip circles in a banded Casson tower, partition the top-stage dotted circles into equivalence classes, where two circles are equivalent if they are connected by a band, choose one dotted circle from each equivalence class, and take meridians to each of these dotted circles.

 \begin{figure}[htb!]
\centering
\begin{subfigure}[b]{0.2\textwidth}
            \begin{overpic}[width=\textwidth, 
 unit=1mm, tics=5]{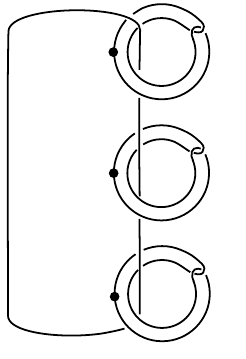}
 \end{overpic}
        \end{subfigure}\hspace{10mm}
        \begin{subfigure}[b]{0.2\textwidth}
            \begin{overpic}[width=\textwidth, 
 unit=1mm, tics=5]{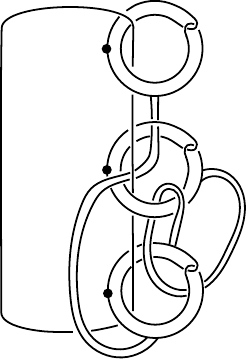}
            \end{overpic}
        \end{subfigure}\hspace{10mm}
        \begin{subfigure}[b]{0.2\textwidth}
            \begin{overpic}[width=\textwidth, 
 unit=1mm, tics=5]{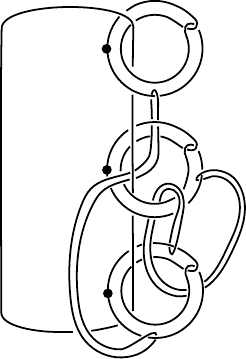}
        \put(70, 40){0} 
        \put(17, 40){0}
            \end{overpic}
        \end{subfigure}
\caption{Left: plumbed 2-handle with three double points. Middle: choice of bands connecting dotted circles. Right: banded plumbed 2-handle. The tip circle is isotopic to a meridian of any of the dotted circles. Attaching a 0-framed 2-handle to any meridian of one of the dotted circles yields a standard 2-handle.}
\label{fig:triangular}
\end{figure}

\begin{proposition}\label{prop:THhomeo}
    The manifold $(TH,\partial_-TH)$ is homeomorphic rel boundary to an open 2-handle.
\end{proposition}

\begin{proof}
    Figure~\ref{fig:TH2} exhibits the triangular kinky handle $Tk$ as a plumbed 2-handle with two positive double points, union a 0-framed 2-handle attached along a curve pairing the double point loops. This curve together with the dashed meridian in Figure \ref{fig:TH2} constitute a triangular family of curves for the plumbed 2-handle. It follows that $TH$ is a banded Casson handle. 
\end{proof}

In general, triangular plumbed handles are not diffeomorphic rel attaching region to plumbed 2-handles (though this can happen, for example, by attaching 2-handles to a subset of the meridians $\mu_i$.) We will show this for $(Tk, \partial_- Tk)$ in Proposition~\ref{prop:notaplumbedhandle}. 
 The question of whether $TH$ is a Casson handle is more subtle and we discuss this in Section~\ref{sec:stein}.

\subsection{The family $BH_m$ of triangular generalized Casson handles}\label{sec:triangulartower}
Using triangular families of curves, we construct generalized Casson handles with good branched covering properties, which do not clearly belong to any of the previously mentioned subclasses. Let $CH_+$ be the linear chain Casson handle with only positive double points \cite{gompfstipsicz}. 
Recall \textit{surface blocks} of disk embedding theory are thickened embedded surfaces (see \cite[Section 13.4.1]{diskembedding} for an excellent source). Let $n\ge 2$ be even and let $\Sigma_n$ be the genus $n$ surface with one boundary component, so $\Sigma_n\times D^2$ is a {surface block}. It has a product-framed attaching region $\partial_- \cong (\partial\Sigma_n)\times D^2$ with attaching circle $(\partial \Sigma_n)\times 0$, and admits a standard Kirby diagram consisting of $2n$ 1-handles. 

    \begin{figure}[htb!]
\centering
\begin{subfigure}[b]{0.3\textwidth}
            \begin{overpic}[width=\textwidth, 
 unit=1mm, tics=5]{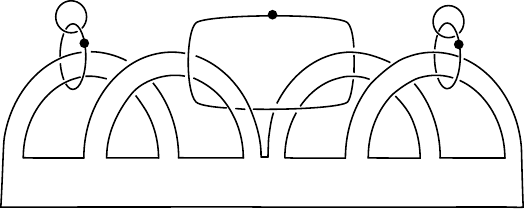}
        \put(90, 35){$CH$} 
        \put(18, 35){$CH$} 
            \end{overpic}
            \caption{}
            \label{fig:bH2}
        \end{subfigure} \hspace{5mm}
        \begin{subfigure}[b]{0.55\textwidth}
            \begin{overpic}[width=\textwidth, 
 unit=1mm, tics=5]{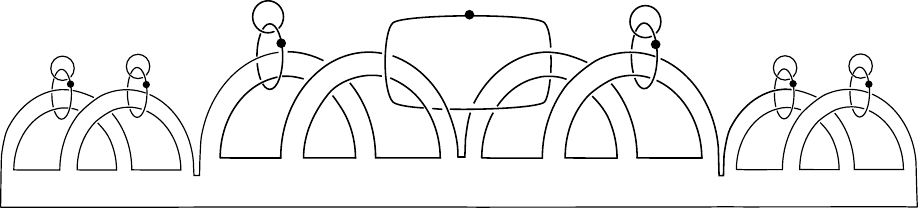}
 \end{overpic}
            \caption{}
            \label{fig:bH}
        \end{subfigure}
\caption{Kirby diagrams for $B(2,CH)$ and $B(4,CH)$; all indicated meridians have 0-framed Casson handles attached. The attaching circle is nullhomotopic in the submanifold where only first stage plumbed handles are attached to the meridians instead of Casson handles.}
\label{fig:bHs}
\end{figure}

Starting with $\Sigma_n\times D^2$, attaching a 0-framed 2-handle running over exactly two of the 1-handles (similarly to Figure~\ref{fig:TH2}), and attaching copies of a fixed choice of Casson handle $CH$ to 0-framed meridians of all other 1-handles, we obtain the manifolds $B(n,CH)$ shown in Figure~\ref{fig:bHs}. Note that each `block' $B(n,CH)$ is made with multiple copies of the same Casson handle $CH$, and has an attaching circle (the long curve in Figure~\ref{fig:bHs})  
and a single tip circle (a meridian to the center 1-handle in Figure~\ref{fig:bHs}) which are both given the blackboard framing. Each $B(n,CH)$ has fundamental group $\bbZ$ generated by the tip circle, and the attaching circle is nullhomotopic in the submanifold consisting of $\Sigma_n\times D^2$ union the triangular 2-handle, union the first stage plumbed handles of the Casson handles 
(after sliding pairs of strands over plumbed handles, the attaching circle becomes homotopic to a curve which runs over none of the dotted circles.) 
Attaching a $k$-framed 2-handle along the large curve yields a Kirby diagram of $B(n,CH)$ attached to $B^4$ along a $k$-framed unknot.

 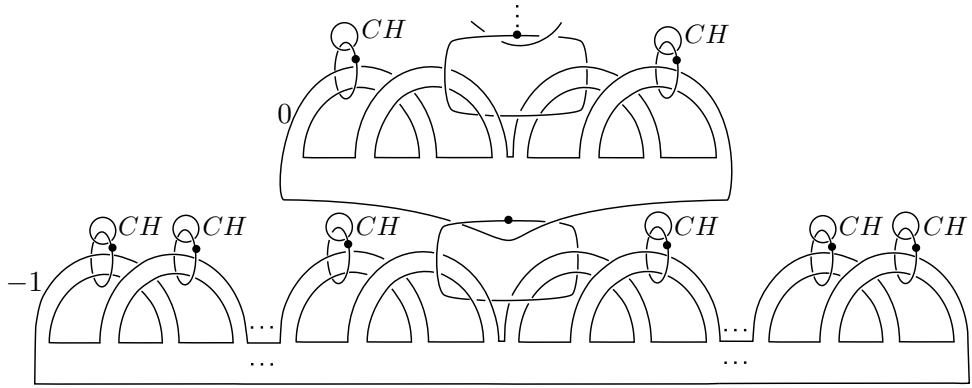
\begin{figure}[ht]
\centering
\begin{subfigure}[b]{0.75\textwidth}
            \begin{overpic}[width=\textwidth, 
 unit=1mm, tics=5]{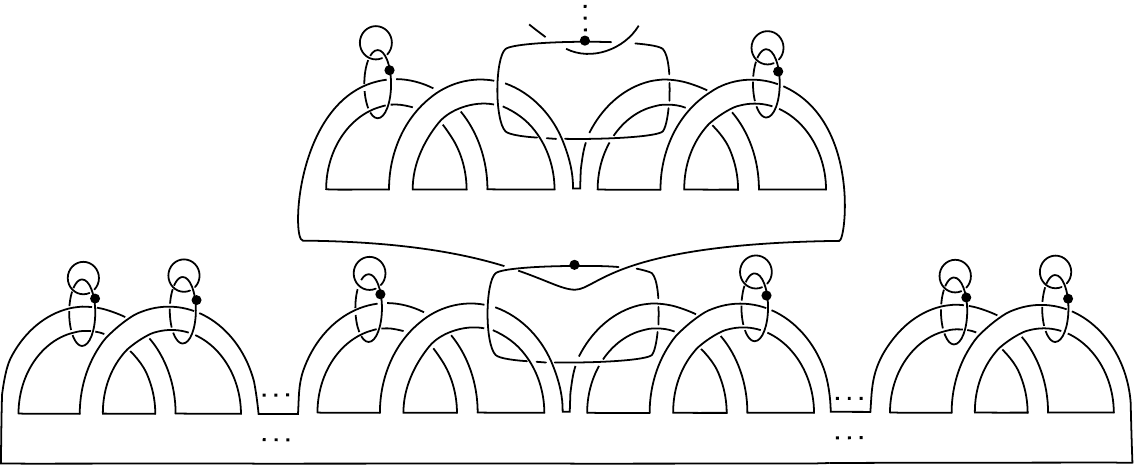}
 \put(9, 16){\small$CH$} 
 \put(18, 16){\small$CH$}
 \put(34.4, 16){\small$CH$}
 \put(68.4, 16){\small$CH$}
 \put(86, 16){\small$CH$} 
 \put(95, 16){\small$CH$}
 \put(35, 37){\small$CH$}
 \put(69.5, 36.5){\small$CH$}

  \put(-3, 10){$-1$}
 \put(26, 28){$0$}
 \end{overpic}
            \end{subfigure}
\caption{Partial Kirby diagram of $\M_m$. Changing the $-1$-framing to 0 gives $\A_m$. All Casson handles are $CH_+$ attached with the 0-framing.}
\label{fig:branchset1}
\end{figure}
\begin{figure}[ht]
\centering
\begin{subfigure}[b]{0.43\textwidth}
            \begin{overpic}[width=\textwidth, 
 unit=1mm, tics=5]{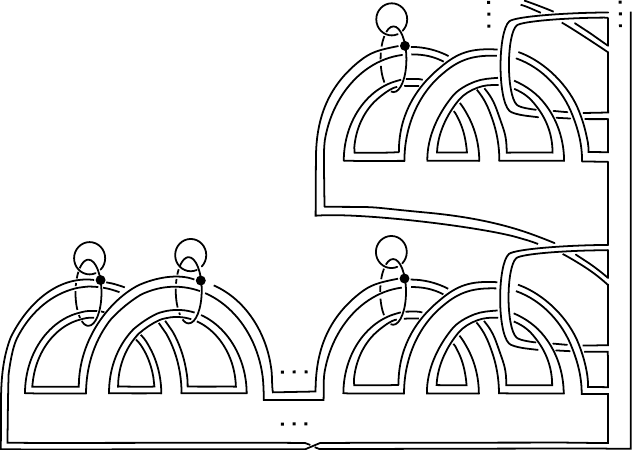}
 \put(17, 31){\small$CH$}
 \put(34, 31){\small$CH$}
 \put(65, 31){\small$CH$}
 \put(65, 68){\small$CH$}
 \end{overpic}
            \end{subfigure} \hspace{5mm}
            \begin{subfigure}[b]{0.46\textwidth}
            \begin{overpic}[width=\textwidth, 
 unit=1mm, tics=5]{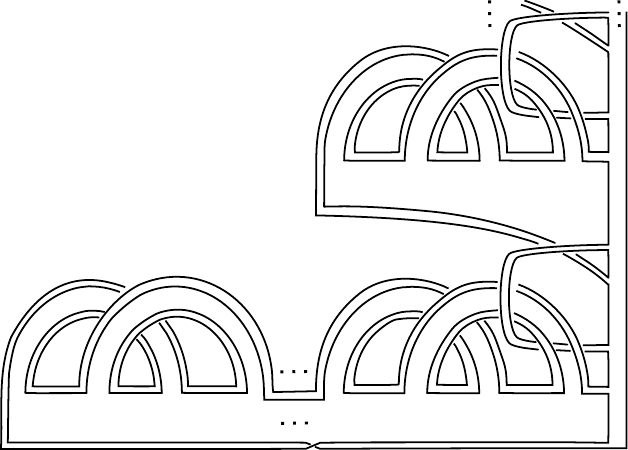}
 \end{overpic}
            \end{subfigure}
\caption{Left: Partial branch set of $\M_m$ in the standard $\R^4$ (removing the half-twist at the bottom gives the branch set of $\A_m$.) Right: homeomorphic image of the same branch set in the standard $\R^4$. All Casson handles are $CH_+$.}
        \label{fig:branchset2}
\end{figure}
\begin{figure}[htb!]
\centering
\begin{subfigure}[b]{0.5\textwidth}
            \begin{overpic}[width=\textwidth, 
 unit=1mm, tics=5]{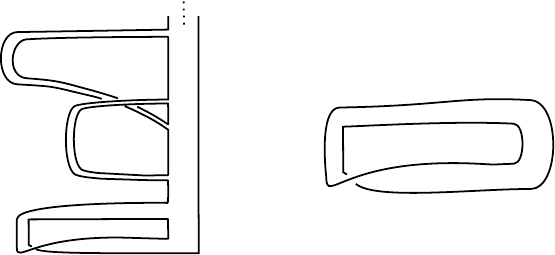}
 \end{overpic}
            \end{subfigure}
\caption{Isotopy showing the branch set of $\M_m$ is standard.}
        \label{fig:branchset3}
\end{figure}

Any infinite stack of the blocks $B(n,CH)$, where each attaching region is attached to the previous tip region with the 0-framing, and various choices of $n$ and Casson handle are used, is a generalized Casson handle. It admits an exhaustion as in Definition \ref{def:gch} by letting $n_1$ be the first copy of $\Sigma_n\times D^2$ union its triangular 2-handle, union its first layer of plumbed handles; letting $n_2$ be the second stage $\Sigma_n\times D^2$ union triangular 2-handle, union its first layer of plumbed handles, union the second layer of plumbed handles for the previous stage; and so on.  Such infinite towers of $B(n,CH)$ have $\pi$-rotational symmetry about the `stacking' axis as long as the Casson handles used enjoy the correct rotational symmetry. 
   
\begin{definition}[The handles $BH_m$]\label{def:BHm}
For $m\ge 2$ even, define the generalized Casson handle $BH_m$ to be an infinite union of $B(n,CH)$ as described above, where the first stage is $B(m,CH)$, all other stages are $B(2,CH)$, and all Casson handles are $CH_+$, i.e. the linear chain Casson handle with all positive double points. 
\end{definition}
Using the $BH_m$, we define families of manifolds $\M_m$ and $\A_m$. These will turn out to be exotic $\overline{\bbC P^2}\setminus pt$'s and $S^2\times \R^2$'s which are branched double covers of exotic surfaces in $\R^4$. In Theorem \ref{thm:exoticmobius} we will show these manifolds are not diffeomorphic. For now we verify the following.

\begin{proposition}\label{prop:mobius}
    Let $m\ge 2$ be even. Let $\M_m$ (resp. $\A_m$) denote $B^4$ union $BH_m$ attached along a $-1$-framed (resp. $0-$framed) unknot, minus the boundary. Then $\M_m$ (resp. $\A_m)$ is homeomorphic to $\overline{\bbC P^2}\setminus pt$ (resp. $S^2\times \R^2$), and is a double branched cover of $\R^4$ over a topologically standard proper open M\"obius strip (resp. topologically standard proper open annulus). 
\end{proposition}
A partial Kirby diagram of $\M_m$ is shown in Figure~\ref{fig:branchset1}. By topologically standard, we mean write $\R^4 = \R^3\times \R$, take an embedded copy of the compact unknotted embedded surface in $\R^3\times\{0\}$, and adjoin a standard proper collar $(\R^3, \text{boundary})\times [0,\infty)$.

\begin{proof}
The homeomorphisms follow from Theorem \ref{thm:open2handles} since $BH_m$ are homeomorphic rel boundary to open 2-handles. 

Note $CH_+$ has $\pi$-rotational symmetry. Quotienting $B^4$ union $BH_m$ by its $\pi$-rotational symmetry yields a branch surface in $\R^4_{std}$, indicated in the left of Figure~\ref{fig:branchset2} for the case of $\M_m$; note, this is a nonstandard Kirby diagram of $\R^4_{std}$ obtained as an infinite end-sum of copies of the interior of $S^1\times B^3$ union $CH_+$, which are known to be $\R^4_{std}$. The branch surface is a properly embedded (open) annulus or M\"obius strip depending on the attaching framing (appearing as a half-twist in the lower band of the surface). Replacing the Casson handles with 0-framed 2-handles amounts to applying a homeomorphism to $\R^4$, hence does not affect the topological isotopy class of the branch set. After this is done, the surface becomes smoothly ambiently isotopic to the standard M\"obius strip (or annulus) in $\R^4$, as indicated by Figures \ref{fig:branchset2} and \ref{fig:branchset3}. 
\end{proof}

    \section{Stein structures and genus bounds}\label{sec:stein}

We begin by realizing the exotic $\R^4$, $\cR$, of Figure~\ref{fig:exoticR4}, and the manifolds $\M_m$ and $\A_m$ of Proposition~\ref{prop:mobius} as Stein manifolds. We then apply a Stein adjunction inequality to $TH$, $\M_m$, and $\A_m$, proving Theorem \ref{thm:mobius}. Note that $\cR$ is made by attaching $(TH,\partial_-TH)$ to the disk complement $X_{-1}$ of Figure~\ref{fig:diskcomplement3}. 
The handle $TH$ is Stein in the sense that it can often be attached to other 4-manifolds to give Stein manifolds.
\begin{theorem}\label{thm:stein}
    The manifold $\cR$ is Stein, and $(TH,\partial_-TH)$ is a Stein handle.
\end{theorem}
\begin{proof}
    By \cite{gompfstipsicz} it suffices to describe $\cR$ as a Kirby diagram in standard form, where all 2-handles are attached along Legendrian knots $L$ with framing $tb(L)-1$.  
    Figures \ref{fig:steindisk} and \ref{fig:steinpiece} show how to Stein realize the disk complement $X_{-1}$ and the triangular kinky handle $(Tk,\partial_-Tk)$ respectively. To Stein realize their unions we record the relevant smooth attaching curves, drawn in blue. 
    \begin{figure}[t!]
    \centering
        \begin{subfigure}[b]{0.2\textwidth}
            \begin{overpic}[width=\textwidth, 
 unit=1mm, tics=5]{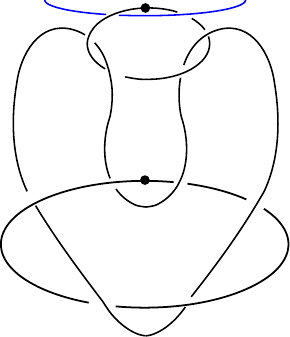}
 \put(75, 90){$-1$}
 \end{overpic}
            \caption{}
        \end{subfigure} \hspace{5mm}
        \begin{subfigure}[b]{0.23\textwidth}
            \begin{overpic}[width=\textwidth, 
 unit=1mm, tics=5]{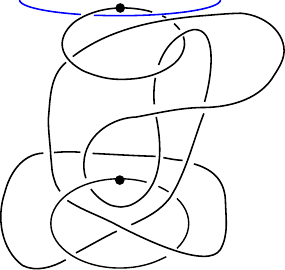}
             \put(90, 91){$-1$}
 \end{overpic}
            \caption{}
        \end{subfigure}\hspace{5mm}
        \begin{subfigure}[b]{0.22\textwidth}
            \begin{overpic}[width=\textwidth, 
 unit=1mm, tics=5]{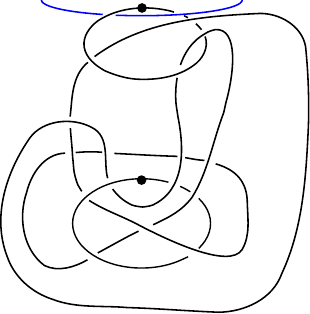}
             \put(95, 92){$-1$}
 \end{overpic}
            \caption{}
            \label{fig:steindiskbeforeslide}
        \end{subfigure}\hspace{5mm}
        \begin{subfigure}[b]{0.22\textwidth}
            \begin{overpic}[width=\textwidth, 
 unit=1mm, tics=5]{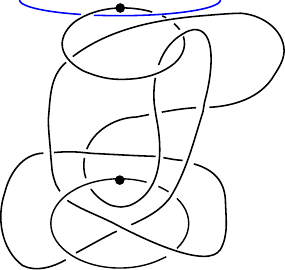}
             \put(95, 88){$-3$}
 \end{overpic}
            \caption{}
            \label{fig:steindiskafterslide}
        \end{subfigure}
        
        \begin{subfigure}[b]{0.22\textwidth}
            \begin{overpic}[width=\textwidth, 
 unit=1mm, tics=5]{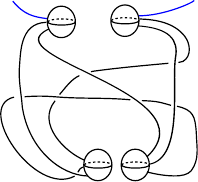}
             \put(91, 75){$-3$}
 \end{overpic}
            \caption{}
        \end{subfigure}\hspace{10mm}
        \begin{subfigure}[b]{0.25\textwidth}
            \begin{overpic}[width=\textwidth, 
 unit=1mm, tics=5]{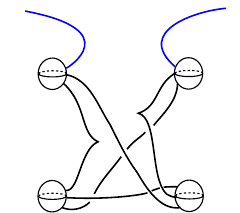}
             \put(65, 30){$-3$}
 \end{overpic}
            \caption{}
            \label{fig:finaldiskcomplement}
        \end{subfigure}
       \caption{Showing the disk complement $X_{-1}$ in the construction of $\cR$ is Stein. The blue curves are used to track smooth attaching data. 
       }
        \label{fig:steindisk}
    \end{figure}
 \begin{figure}[ht!]
    \centering
        \begin{subfigure}[b]{0.18\textwidth}
        \includegraphics[width=\textwidth]{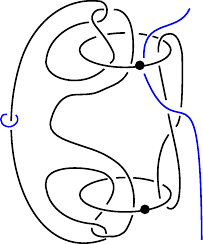}
            \caption{}
        \end{subfigure}
        \hspace{5mm}
        \begin{subfigure}[b]{0.12\textwidth}
            \includegraphics[width=\textwidth]{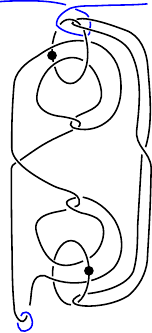}
            \caption{}
        \end{subfigure} 
        \hspace{5mm}
        \begin{subfigure}[b]{0.12\textwidth}
            \includegraphics[width=\textwidth]{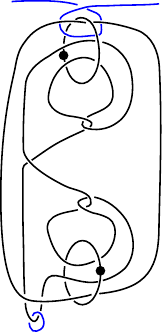}
            \caption{}
        \end{subfigure}\hspace{5mm}
        \begin{subfigure}[b]{0.15\textwidth}
            \includegraphics[width=\textwidth]{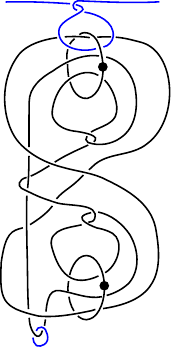}
            \caption{}
        \end{subfigure} \hspace{5mm}
          \begin{subfigure}[b]{0.15\textwidth}
            \includegraphics[width=\textwidth]{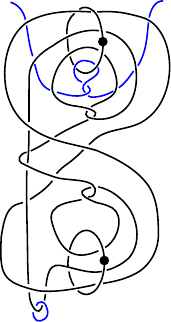}
            \caption{}
        \end{subfigure}
      
        \begin{subfigure}[b]{0.18\textwidth}
        \includegraphics[width=\textwidth]{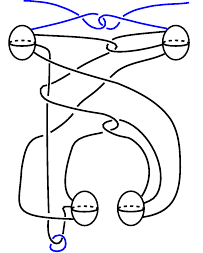}
            \caption{}
        \end{subfigure}
        \hspace{10mm}
        \begin{subfigure}[b]{0.2\textwidth}
        \includegraphics[width=\textwidth]{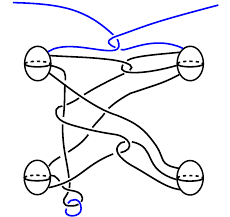}
            \caption{}
        \end{subfigure}
        \hspace{10mm}
        \begin{subfigure}[b]{0.2\textwidth}
        \includegraphics[width=\textwidth]{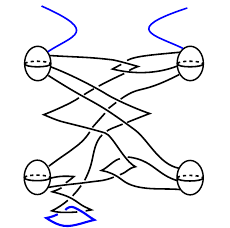}
            \caption{}
            \label{fig:steinpieceuseful}
        \end{subfigure}
       \caption{Kirby diagrams of $(Tk,\partial_-Tk)$. The blue curves track smooth attaching data. In (g)-(h), we (fully) untwist the top blue curve, which does not affect the smooth type of the attachment.}
        \label{fig:steinpiece}
    \end{figure}
In Figure \ref{fig:steindisk} all moves are diagram isotopies except going from Figure \ref{fig:steindiskbeforeslide} to \ref{fig:steindiskafterslide}, which shows the 2-handle sliding under the lower 1-handle. In Figure~\ref{fig:finaldiskcomplement} the black curve has 
    $tb = -2$, thus the required $-3$-framed attachment is Stein. Figure \ref{fig:steinpiece} tracks attaching data of $Tk$ through isotopy. 

Omitting the lower (blue) meridian of Figure \ref{fig:steinpieceuseful} and Legendrian realizing as in Figure \ref{fig:steinpiecepart2a}, we see an `X' shaped 2-handle with $tb=1$ and the 2-handle corresponding to the attaching curve of $Tk$, which has $tb = 3$. Stacking $Tk$ (Figure~\ref{fig:steinpieceuseful}) on top of $X_{-1}$ (Figure~\ref{fig:finaldiskcomplement}) yields Figure \ref{fig:steinpiecepart2b} (where only the top 1-handle of $X_{-1}$ is shown); the `X' shaped 2-handle has $tb=1$ and the other 2-handle has $tb = 3$. Stabilizing the 2-handle which has $tb = 3$ twice results in Figure \ref{fig:steinpiecepart2c}, where both 0-framed attachments are Stein. Attaching $Tk$ to $Tk$ using the blue curves also yields Figure~\ref{fig:steinpiecepart2b}, hence the desired 0-framed attachment can be performed by Stein handle attachments.
    \end{proof}

    \begin{figure}[htb!]
        \centering
         \begin{subfigure}[b]{0.28\textwidth}
        \includegraphics[width=\textwidth]{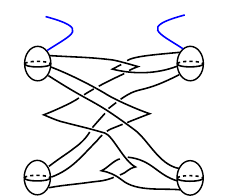}
            \caption{}
            \label{fig:steinpiecepart2a}
        \end{subfigure} \hspace{8mm}
        \begin{subfigure}[b]{0.28\textwidth}
            \includegraphics[width=\textwidth]{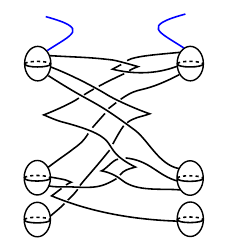}
            \caption{}
            \label{fig:steinpiecepart2b}
        \end{subfigure} \hspace{10mm}
        \begin{subfigure}[b]{0.28\textwidth}
        \includegraphics[width=\textwidth]{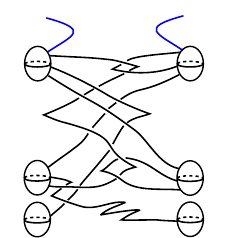}
            \caption{}
\label{fig:steinpiecepart2c}\end{subfigure}
        
       \caption{In all three figures, blue (smooth) curves are used to track how subsequent $Tk$ units are attached. Three Stein realizations: $Tk$ attached to $B^4$ along a (2-framed) unknot; then $Tk$ attached to the upper 1-handle of $X_{-1}$, or attached to the upper 1-handle of $Tk$; then the stabilized version so all smooth 0-framed attachments are Stein. }
        \label{fig:steinpiecepart2}
    \end{figure}
  \begin{figure}[tb]
    \centering
        \begin{subfigure}[b]{0.15\textwidth}
            \begin{overpic}[width=\textwidth, 
 unit=1mm, tics=5]{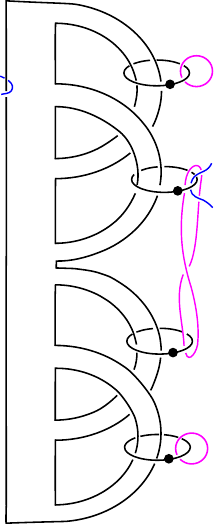}
 \end{overpic}
            \caption{}
        \end{subfigure} \hspace{5mm}
        \begin{subfigure}[b]{0.15\textwidth}
            \begin{overpic}[width=\textwidth, 
 unit=1mm, tics=5]{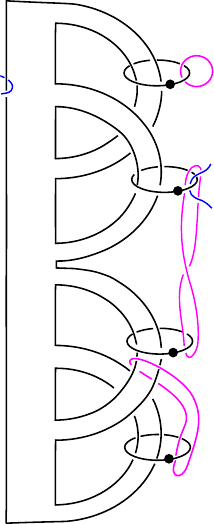}
 \end{overpic}
            \caption{}
        \end{subfigure}\hspace{5mm}
        \begin{subfigure}[b]{0.165\textwidth}
            \begin{overpic}[width=\textwidth, 
 unit=1mm, tics=5]{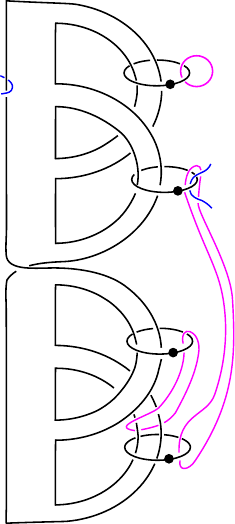}
 \end{overpic}
            \caption{}
            \label{fig:fff}
\end{subfigure}\hspace{5mm}\begin{subfigure}[b]{0.28\textwidth}
            \begin{overpic}[width=\textwidth, 
 unit=1mm, tics=5]{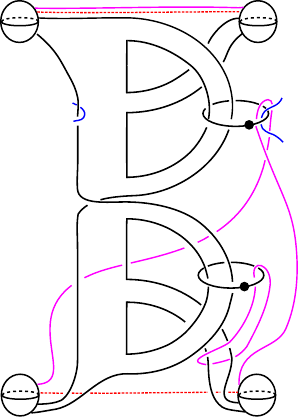}
 \end{overpic}
            \caption{}
            \label{fig:ff}
        \end{subfigure}
       \caption{Isotoping $B(2,\cdot)$ including attaching data.
       }
        \label{fig:steinb2part1}
    \end{figure}
\begin{theorem}\label{thm:chern}
    Let $m \ge 2$ be even. The manifolds $\M_m$ and $\A_m$ of Proposition \ref{prop:mobius} admit Stein structures such that if $[\Sigma_{\M_m}]$ and $[\Sigma_{\A_m}]$ denote respective generators of $H_2$, then we have $|\langle c_1(\M_m),[\Sigma_{\M_m}] \rangle| = 2m-1$ and $|\langle c_1(\A_m),[\Sigma_{\A_m}] \rangle| =2m-2$.
\end{theorem}
\begin{proof}
    We first Stein realize the infinite unions of $B(n,CH_+)$ \textit{without} their Casson handles attached. This is done by constructing standard form handlebodies as before. Tracking attaching circles of the Casson handles shows they can all be taken to have $tb = 0$. Then, by Gompf's proof of \cite[Theorem 3.1]{gompfsteinhandlebody} (see also \cite[Theorem 2.4]{gompfgenera} for another explanation), attaching 0-framed copies of $CH_+$ along these curves yields a Stein manifold (we only need one positive double point in the first stage to correct $tb = 0$, and further stages can always be taken to have a single positive double point). 
    Note that the inclusion induced by adding the Casson handles is of complex manifolds, and the Casson handles do not affect $H_2$. By pulling back $c_1$, we find the action of $c_1$ on a generator of $H_2$ is the same as before the Casson handles were attached, and can be computed as a rotation number in the standard form handlebody of the first step \cite[Proposition 2.3]{gompfsteinhandlebody}.

    To build the first part, recall $B(n,CH_+)$ without its Casson handles (denoted $B(n,\cdot)$) is a disk bundle over a genus $n$ surface, union a single 0-framed 2-handle. First consider the $n= 2$ case. In Figure~\ref{fig:steinb2part1} we begin with a disk bundle over the genus 2 surface, drawn together with the attaching data needed to reconstruct $B(2,CH_+)$ (shown as pink curves) and the attaching data needed to combine multiple units (shown as blue curves). We will isotop this disk bundle into its standard Stein form \cite{gompfsteinhandlebody}, bringing along the attaching curves. This is started in Figure \ref{fig:steinb2part1} and continued in Figure \ref{fig:steinb2part2} for the case of $B(2,\cdot)$. The dashed red lines are used to track framings; note we often slide 2-handles under 1-handles when the 2-handle runs over the 1-handle algebraically 0 times, which does not change the smooth framing. Completed diagrams are in Figure \ref{b2noattach} for the \textit{first-stage} version of $B(2,\cdot)$ (i.e. the copy of $B(2,\cdot)$ that contains the attaching circle of $BH_2$) and Figure \ref{fig:b2attach} for the higher-stage version. In both, the triangular 2-handle has $tb = 1$ as required and the Casson handle attaching circles have $tb=0$. The large 2-handle has $tb = 1$ in the higher-stage version, and has $tb = 3$ (and $r = 0$) when attached to $B^4$. It is straightforward to see that when two units are combined using the attaching curves, the second stage attaching 2-handle can be taken to have $tb = 1$. This yields a standard form handlebody for $BH_2$ without its Casson handles. In order to attach $BH_2$ to $B^4$ with smooth 0 or $-1$ framing, we must (homogeneously) stabilize the first curve (the black 2-handle in Figure \ref{b2noattach}) 2 or 3 times, respectively. Since this 2-handle generates homology, after attaching Casson handles we find $|\langle c_1(M_2), [\Sigma_{M_2}]\rangle| = 3$ and $\langle|c_1(A_2), [\Sigma_{A_2}]\rangle| = 2$.
    
    For $n\ge 4$, note $B(n,CH_+)$ only appears in $BH_n$ as the first stage, so we only need to track how $B(2,CH_+)$ is attached on top. By swinging genus 1 blocks of $B(n,CH_+)$ around as in Figure~\ref{fig:swing}, and proceeding similarly to the genus 2 case, we obtain Figure~\ref{fig:steinBn}. Note the black 2-handle has $tb = 2n - 1$ and $r=0$. Thus, we must stabilize it $2n-1$ or $2n-1$ times, to attach $BH_n$ to $B^4$ with 0- or $-1$- framing. This yields the claimed rotation numbers. 
\end{proof}
\begin{figure}[htb!]
    \centering
            \begin{overpic}[width=0.6\textwidth, 
 unit=1mm, tics=5]{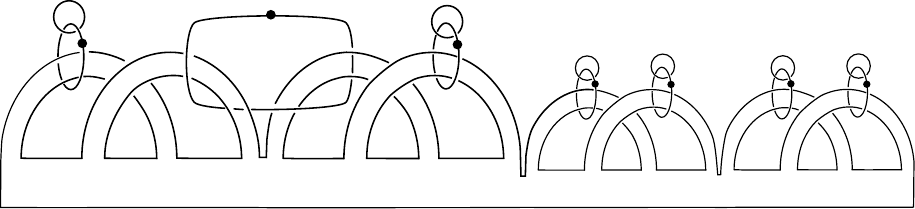}
 \end{overpic}
            \caption{Result of swinging around genus 1 blocks in $B(4,CH_+)$.}
            \label{fig:swing}
        \end{figure}
 \begin{figure}[htb!]
    \centering
        \begin{subfigure}[b]{0.29\textwidth}
            \begin{overpic}[width=\textwidth, 
 unit=1mm, tics=5]{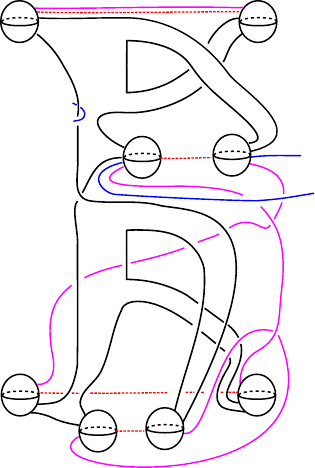}
 \end{overpic}
            \caption{}
        \end{subfigure}\hspace{5mm}
         \begin{subfigure}[b]{0.28\textwidth}
            \begin{overpic}[width=\textwidth, 
 unit=1mm, tics=5]{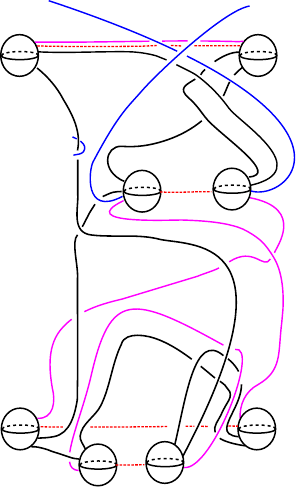}
 \end{overpic}
            \caption{}
        \end{subfigure}\hspace{5mm}
         \begin{subfigure}[b]{0.24\textwidth}
            \begin{overpic}[width=\textwidth, 
 unit=1mm, tics=5]{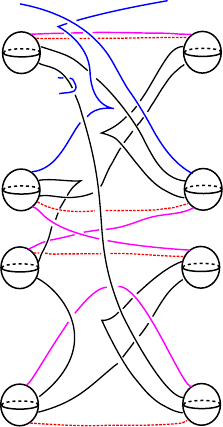}
 \end{overpic}
            \caption{}
        \end{subfigure}\hspace{10mm}
        \begin{subfigure}[b]{0.28\textwidth}
            \begin{overpic}[width=\textwidth, 
 unit=1mm, tics=5]{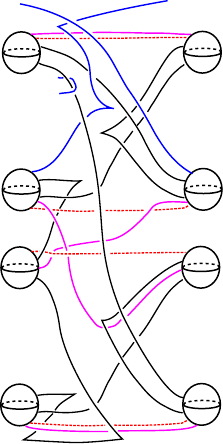}
 \end{overpic}
            \caption{}
            \label{b2noattach}
        \end{subfigure}\hspace{10mm}
        \begin{subfigure}[b]{0.28\textwidth}
            \begin{overpic}[width=\textwidth, 
 unit=1mm, tics=5]{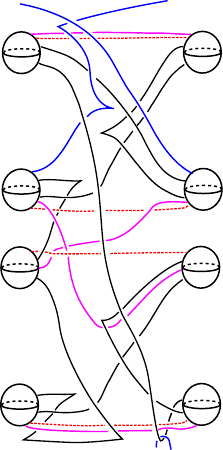}
 \end{overpic}
            \caption{}
            \label{fig:b2attach}
        \end{subfigure}\hspace{10mm}
        \begin{subfigure}[b]{0.2\textwidth}
            \begin{overpic}[width=\textwidth, 
 unit=1mm, tics=5]{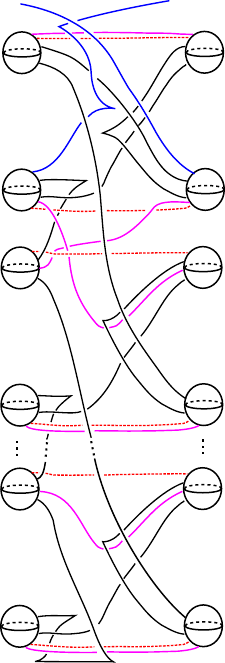}
 \end{overpic}
            \caption{}
\label{fig:steinBn}
        \end{subfigure}
       \caption{(a)-(e): Stein realizing $B(2,\cdot)$ continued. Note (b) to (c) does a $2\pi$ twist of the top blue curve. (f): Stein realizing $B(n,\cdot)$. Pink meridians to 1-handles are attaching circles for Casson handles.
       }
        \label{fig:steinb2part2}
    \end{figure}

Recall the Stein adjunction inequality of Lisca-Mati\'c \cite{liscamatic}, which also appears as Theorem 11.4.7 of Gompf-Stipsicz \cite{gompfstipsicz}:
\begin{theorem}[\cite{gompfstipsicz, liscamatic}]\label{thm:adjunction}
    Let $\Sigma$ be a connected orientable smooth embedded surface with $g(\Sigma)>0$ in a Stein surface $S$. Then
    \[2g(\Sigma) -2 \ge \Sigma\cdot \Sigma + |\langle c_1(S),[\Sigma]\rangle|\]
\end{theorem}
Theorem~\ref{thm:adjunction} allows us to determine the slice genus of the attaching circle of $TH$, in $TH$.

\begin{proposition}\label{prop:notaplumbedhandle}
    The smooth slice genus of the attaching circle of $(TH,\partial_-TH)$ in $TH$ is 2. Thus $TH$ is an exotic open 2-handle, and $(Tk, \partial_- Tk)$ is not any plumbed 2-handle.
\end{proposition}

\begin{proof}
Let $S$ be obtained by attaching $TH$ to $B^4$ along a 0-framed unknot. After (homogeneously) stablizing the attaching Legendrian in Figure \ref{fig:steinpiecepart2a} twice to drop its $tb$ to 1, we see $S$ admits a Stein structure. Let $\Sigma$ be the smooth surface obtained by capping a Seifert surface for the attaching Legendrian with a core of the 2-handle generating $H_2(S)$. Then $[\Sigma]$ generates $H_2(S)$, $\Sigma\cdot \Sigma =0$, and by considering rotation number in a standard form handlebody we have $|\langle c_1(S),[\Sigma] \rangle| =2$. By the adjunction inequality above it follows that for any connected smooth embedded orientable surface $F$ generating $H_2(S)$, $g(F)\ge 2$, and the obvious Seifert surface shows the minimal genus equals 2. By considering connected, compact, oriented surfaces with circle boundary $(\Sigma,\partial \Sigma = S^1)\hookrightarrow (TH,\partial_-TH)$, union a fixed disk in $B^4$ bounded by the unknot, it follows that the genus of the attaching circle in $TH$ is 2.

If $Tk$ was a plumbed 2-handle, then it must have exactly one double point by considering the fundamental group. In the plumbed 2-handle, the minimal genus of a smoothly embedded surface bounded by this circle is 1 (apply the Stein adjunction inequality to the plumbed handle attached to $B^4$ along a 0-framed unknot). 
\end{proof}

\begin{remark} We determined, using KnotJob \cite{knotjob}, that the generalization of Rasmussen's $s$-invariant due to Manolescu, Marengon, Sarkar, and Willis \cite{ManolescuMarengon} provides a lower bound of 1 on the desired genus in $Tk$. 
    \end{remark}

\begin{conjecture}
    $TH$ is not a Casson handle.
\end{conjecture}

Necessarily, it would have to be a Stein, slice genus 2 Casson handle. 
Observe $TH$ is a union of subsets with fundamental group $\bbZ$ relative to their attaching region; it seems such an exhaustion should force $TH$ to be a linear chain Casson handle, which would then have slice genus at most 1. However it is unknown whether a Casson handle can arise from different signed trees. The end space of a tree used to build a Casson handle is not known to be an invariant of the Casson handle: the complication arises since when the upper boundary is deleted, all ends are collapsed to a point.

\begin{question}[See also Gompf, {\cite[Remark after Prop. 5.1]{gompfproper}}]
    Suppose $CH$ is a Casson handle formed from a given signed tree, i.e. by stacking self-plumbed 2-handles according to the tree, then deleting all boundary except the open attaching region. Can $CH$ arise from a different signed tree?
\end{question}

In the following, a \textit{small exotic plane} is one whose double branched cover is a small exotic $\R^4$, i.e. one whose compact codimension-0 submanifolds embed in $\R^4$. For background on end-sums, see \cite{gompfinfinite, gompfproper}.

\begin{theorem}\label{thm:exoticmobius}
    Let $m\ge 2$ be even. The manifolds $\A_m$ and $\M_m$ are infinite families of exotic smoothings of $S^2\times \R^2$ and $\overline{\bbC P^2}\setminus pt$ respectively, distinguished by the minimal genus of a smooth surface generating $H_2$. Thus, their branch sets are infinite families of topologically but not smoothly isotopic proper open annuli and M\"obius strips in $\R^4$. Moreover, the branch sets are not end sums of smoothly standard surfaces with \textbf{any} exotic planes, and remain smoothly nonisotopic when end-summed with any \textbf{small} exotic planes. Letting $A_n$ denote the annuli, $A_n\natural P$ is not isotopic to $A_k$ for any $n,k$ and \textbf{any} exotic plane $P$.
\end{theorem}
\begin{proof}
    We proved the homeomorphisms and that the branch sets are topologically standard in Proposition~\ref{prop:mobius}. By Theorems~\ref{thm:chern} and \ref{thm:adjunction}
    we find
    \[2g(\Sigma(\A_m)) \ge 0 + 2m -2 + 2 = 2m\]
    for any smooth surface  $\Sigma(\A_m)$ generating $H_2(\A_m)$, and
    \[2g(\Sigma(\M_m)) \ge -1 + 2m - 1 + 2 = 2m\]
    for any smooth surface  $\Sigma(\M_m)$ generating $H_2(\M_m)$. By taking the obvious Seifert surfaces we conclude the minimal genus of a smooth generator of $H_2(\M_m)$ or $H_2(\A_m)$ is $m$.
    This establishes exoticness of the branched covers, thus the branch sets are not smoothly properly isotopic. Note that if the surfaces were end sums of standard ones with exotic planes, their branched double covers would be the standard $S^2\times \R^2$ or $\overline{\bbC P^2}\setminus pt$ end-summed with exotic $\R^4$'s (possibly even the standard $\R^4$), and thus have $H_2$ generated by a smooth sphere. 
    
    To see that our surfaces remain smoothly nonisotopic after end-summing with small exotic planes, note that end-summing the surfaces and then taking double branched covers yields $\A_m \natural R$ and $\M_m\natural R$, where $R$ is a small exotic $\R^4$. It suffices to show that end-summing with small exotic $\R^4$'s does not change the genus function: it is clear that end-summing with an $\R^4$ cannot increase the genus function, so suppose it decreases the genus of some connected compact smooth orientable surface $S$ representing a homology class. Recall that to end-sum two manifolds together, open regular neighborhoods of proper rays are deleted, and the resulting two $\R^3$ boundary components are identified: we call the image of this $\R^3$ under the identification the \textit{end-summing $\R^3$}. Let $U\cong \R^3\times \R$ be a tubular neighborhood of the end-summing $\R^3$ in $A_m\natural R$, and let $\gamma = \{0\} \times\R \subset U$ be a transverse arc. We may assume $S\cap U$ lies in a tubular neighborhood $\nu(\gamma)$. The subset of $S\setminus U$ contained in $R$ is contained in a compact subset of $R$, hence in a compact topological 4-ball $B\subset R$. By the connectedness of $S$, we may assume $S\cap \nu(\gamma)\cap \partial B\ne \emptyset$. The interior $R'$ of $B$ is an exotic $\R^4$ which embeds in $\R^4$ standard (since $B$ embeds in some smooth compact submanifold of $R$, which all embed in $\R^4$ by assumption). By Quinn's stable homeomorphism theorem \cite{endsofmaps}, we may assume that $\partial B$ is smooth in a neighborhood $V \cong \R^3 \subset \partial B$ of some point $x \in \partial B \cap \nu(\gamma)$. By shrinking $\nu(\gamma)$ radially, we may assume $S\cap \partial B$ is contained in $V$. Since homotopy implies isotopy for curves in a 4-manifold, we may isotop $\nu(\gamma)$ so that $\nu(\gamma)\cap B = V \times [0,1)$, a collar of $V$ in $B$. Thus, $S$ is contained in $A_m\natural R'$, where the end-sum is performed along the shrunk $\nu(\gamma)$. Then since $R'\hookrightarrow \R^4$, there exists an embedding $A_m\natural R'\hookrightarrow A_m\natural \R^4 = A_m$ inducing the identity map on homology, successfully lowering the genus of the homology class represented by $S$, a contradiction.

    To verify the final claim about $A_m$, it remains to check the case where $P$ is a large exotic plane, i.e. has a large exotic $\R^4$ as branched double cover. Since $BH_m$ embeds in a 2-handle rel attaching region, $A_m$ embeds in $S^2\times \R^2$, hence in $\R^4$. But $A_m \natural R'$ for any large exotic $\R^4$ $R'$ does not embed in $\R^4$, and hence is not diffeomorphic to $A_n$, for any $m,n$. 
\end{proof}

\begin{remark}
   We thank Bob Gompf for pointing out that end-summing with small $\R^4$'s does not change the genus function (and for sketching the above argument for this), and that end-summing with large $\R^4$'s can decrease the genus function. In particular, it is possible that $\M_2$ is an end-sum of $\M_m$ with some large exotic $\R^4$, where $m >2$.
\end{remark}

\begin{proof}[Proof of Theorem~\ref{thm:mobius}]
    The theorem is an immediate consequence of Theorem~\ref{thm:exoticmobius}.
\end{proof}

\begin{remark}
  Gompf's exotic proper \textit{half-open} annuli \cite{gompfproper}  are clearly different than our annuli, which are open; his Theorem \cite[Theorem 4.2]{gompfproper} gives infinitely many exotic proper open annuli and M\"obius bands in $\R^4$ whose branched covers are all end sums of the standard cover with exotic $\R^4$'s, hence, have second homology generated by spheres. Thus, our surfaces are distinct from Gompf's.
\end{remark}

\section{End Floer homology and exotic planes}\label{sec:exotic} Theorem \ref{thm:exotic} of this section shows non-vanishing of end Floer homology under a symplectic end, and stability under blowups. We will use this to distinguish exotic planes by showing their double branched covers are distinct exotic $\bbR^4$'s. The proof relies on the nonvanishing of the Heegaard Floer mixed invariant for symplectic manifolds, and is similar to the argument of Gadgil \cite[Theorem 1.6]{gadgil}. In \cite{gadgil}, Gadgil uses Heegaard Floer homology with twisted coefficients, but we are able to simplify the proof in our case by a grading argument.

\subsection{Tools and background}

We recall some properties of Heegaard Floer homology. Throughout, we work over $\mathbb{F} = \bbZ/2\bbZ$. To a closed oriented 3-manifold $Y$ with a spin$^c$-structure $\s$ one can associate the Heegaard Floer homology groups $\HFm(Y,\s)$, $\HFinf(Y,\s)$, and $\HFp(Y,\s)$,
 which are modules over $\ff[U]$, $\ff[U,U^{-1}]$, and $\ff[U,U^{-1}]/U\cdot \ff[U]$ respectively. These modules fit into a $U$-equivariant long exact sequence \cite{OSfourmanifolds}:
\[\hdots \xrightarrow{} \HFm(Y,\s) \xrightarrow{\iota} \HFinf(Y,\s),\xrightarrow{\pi} \HFp(Y,\s) \xrightarrow{\partial} \hdots\] The most important of these groups for us will be $\HFred$, defined as the $U$-torsion submodule of $\HFm(Y,\s)$, or equivalently using the coboundary map above,
\[\HFred^+(Y,\s) = \coker(\pi)\cong \ker(\iota) = \HFred^-(Y,\s).\]
Spin$^c$ cobordisms $(W,\s):(Y_1,\s|_{Y_1})\to (Y_2,\s|_{Y_2})$ induce homomorphisms of all flavors of Floer homology:

\[F^\circ_{W,\s} : \HF^\circ(Y_1,\s|_{Y_1})\to \HF^\circ(Y_2,\s|_{Y_2}),\]
where $\circ = +, -, \infty$. In general, cobordism maps obey a composition law:
\[F^\circ_{W_2,\s_2} \circ F^\circ_{W_1,\s_1} = \sum_{\s:\,\,\s|_{W_1} = \s_1, \s|_{W_2} = \s_2} F^\circ_{W_2\circ W_1, \s}\]

Recall the definition of the Heegaard Floer 4-manifold invariant \cite{OSfourmanifolds}, which is well-defined over $\mathbb{F}$. Cobordisms with $b_2^+ >1$ induce the zero map on $\HFinf$. Given a smooth closed spin$^c$ 4-manifold $X$ with $b_2^+(X) >1$ there exists a connected, separating 3-manifold $P\subset X$ such that $\delta H^1(P)=0$ (where $\delta$ is the Mayer Vietoris coboundary map) such that both components $U$ and $V$ of $X\setminus P$ have $b^2_+ \ge 1$. The manifold $P$ is called an admissible cut and one always exists. Deleting a 4-ball from each of $U$ and $V$, we obtain maps,
\[\begin{tikzcd}
	&& \HFm(P,\s|_P) && \HFp(P,\s|_P) && \\
	\HFm(S^3,\s_0) && \HFred^-(P,\s|_P) && \HFred^+(P,\s|_P) && \HFp(S^3,\s_0)
	\arrow["\text{quotient}"', from=1-5, to=2-5]
	\arrow["{F^+_U}", from=1-5, to=2-7]
	\arrow["{F_V^-}", from=2-1, to=1-3]
	\arrow["{F^-_{V}}"', from=2-1, to=2-3]
	\arrow["\text{inclusion}"', from=2-3, to=1-3]
	\arrow["\partial^{-1}", "\cong"', from=2-3, to=2-5]
	\arrow["\text{induced}"', from=2-5, to=2-7]
\end{tikzcd}\]
where the induced maps follow from the corresponding $\HFinf$ maps vanishing. The composition of the induced maps with $\partial^{-1}$ is called the mixed map $ F^{mix}_{X,\s}$. 
\begin{definition}[Ozsv\'ath-Szab\'o \cite{OSfourmanifolds}]
    The Heegaard Floer 4-manifold invariant of $X$ with spin$^c$ structure $\s$, $\Phi_{X,\s}$, is defined as the image of a top graded generator $\Theta_-$ of $\HFm(S^3,\s_0)\cong \ff[U]$ under $ F^{mix}_{X,\s}$:
\[\Phi_{X,\s} := F^{mix}_{X,\s}(\Theta_-) \in \HFp(S^3,\s_0).\]
\end{definition}

We recall facts about end Floer homology \cite{elihomlidman, gadgil}. Any noncompact 4-manifold $X$ can be decomposed as a compact codimension-0 submanifold union infinitely many compact cobordisms. One can study the end by taking the corresponding direct limit of Heegaard Floer homology. However, composite cobordism maps are generally sums over spin$^c$ structures, and in a direct system this sum must have only one term. This is achieved by restricting to \textit{admissible cobordisms}. Also, since the direct limit only sees 3-manifolds running off the end of $X$, we may work with any spin$^c$ structure only defined in a neighborhood of the end of $X$.

\begin{definition}[Gadgil, Definition 1.1 \cite{gadgil}]\label{def:admissible}
    A cobordism $W_{ij} : Y_i \to Y_j$ is \textit{admissible} if the inclusion map $H^1(W_{ij})\to H^1(Y_j)$ is surjective. A smooth 4-manifold $X$ with one end is \textit{admissible} if it admits a compact exhaustion $C_i \subset C_{i+1}$  where for all $i < j$ the cobordisms $W_{ij} = C_j \setminus C_i$ are admissible, and the exhaustion is called an \textit{admissible exhaustion}. 
\end{definition}

\begin{definition}[Gadgil, Definition 1.3 \cite{gadgil}]
    An \textit{asymptotic spin$^c$ structure} on $X$ is a spin$^c$ structure defined in the complement of a compact set. Two are the same if they agree on the complement of a compact set.
\end{definition}

Gadgil proves the direct limit is a well defined spin$^c$ diffeomorphism invariant of the end of a manifold. 

\begin{theorem}[Gadgil \cite{gadgil}, Theorem 1.4, Proposition 1.5]
    Let $\{C_i\}$ be a compact, admissible exhaustion of a smooth 4-manifold $X$ with one end, and let $Y_i = \partial C_i$. Let $\s$ be an asymptotic spin$^c$ structure on $X$. The end Floer homology $\HE(X,\s) := \varinjlim HF_{red}(Y_i,\s|_{Y_i})$ is a well defined spin$^c$-diffeomorphism invariant of the end of $X$, independent of the choice of admissible exhaustion. Moreover, $HE(\R^4,\s)=0$ for any choice of $\s$.
\end{theorem}

We need the following lemma about symplectic manifolds, which was discussed in Section 2.1 of \cite{OSgenus}; we provide a proof that is likely well-known to experts. Suppose $M$ is a closed 3-manifold $M$ admitting a taut foliation $\cF = \ker \alpha$, for some $\alpha \in \Omega^1(M)$. By tautness, there is a closed 2-form $\omega\in \Omega^2(M)$ with $\omega|_\cF >0$. Then $\tilde\omega = p^*\omega + \ve d(t\alpha)$ is a symplectic form on $M\times I$, where $\ve>0$ is small and $p:M\times I\to M$ is projection \cite{confoliations}. We refer to any such $\tilde\omega$ as a symplectic structure on $M\times I$ corresponding to the foliation $\cF$ on $M$.

\begin{lemma}\label{lem:symplectic}
    Let $M$ be a closed 3-manifold admitting a taut foliation $\cF$, and let $\Sigma$ be a compact leaf with genus at least 1. Give $M\times I$ a symplectic structure corresponding to the foliation, and suppose $M\times I$ embeds symplectically in a closed symplectic 4-manifold $(Z,\omega)$. Then
    $\langle c_1(\omega),i_*[\Sigma]\rangle = \chi(\Sigma).$
\end{lemma}

\begin{proof}
It suffices to show $Z$ admits a compatible almost complex structure $J$ such that $TZ|_\Sigma$ splits as a sum of the complex bundle $T\Sigma$ and the trivial complex line bundle. Then
$i^*c_1(TZ) = c_1(T\Sigma)$,
and we have
\[\langle c_1(\omega), i_*[\Sigma]\rangle = \langle i^* c_1(\omega), [\Sigma]\rangle = \langle c_1(T\Sigma), [\Sigma]\rangle = \chi(\Sigma). \]
Letting $\alpha$ and $\omega$ be as above, the corresponding symplectic structure on $M\times I$ is then
$\tilde\omega = p^*\omega + \ve d(t\alpha)$ for some $\ve>0$, where $p:M\times I\to M$ is projection. As $\omega|_\Sigma$ is an area form, we may choose a compatible complex structure $j$ on $T\Sigma$. Letting $v$ be the vector field on $M$ with $\iota_v\omega \equiv 0$ and $\alpha(v) \equiv 1$, we have
\[T_pZ = T_p\Sigma \oplus \text{span}\{v, \partial_t\}_p\]
for all $p$ in a neighborhood $W\cong \Sigma\times \bbR^2 \subset Z$ of $\Sigma$. Now define a bundle map $J: TZ|_W \to TZ|_W$ acting as $j$ on the $\Sigma$ directions, acting as
$J(v) = -\partial_t$ and $J(\partial_t) = v$ on the normal directions, and extending linearly.
We have $J^2 =- \text{id}$, and $J$ provides the desired complex bundle splitting along $\Sigma$. 
One may now check that for a vector field basis $v_i$ for $T_pZ$ consisting of $\partial_t$, $v$, and a basis for $T_p\Sigma$, we have $\tilde\omega(Jv_i, Jv_j) = \tilde\omega(v_i,v_j)$ and $\tilde\omega(v_i,Jv_j) >0$, for all $i,j$. Thus, these formulae hold for all vector fields on a neighborhood of $p$. 
Since we have a compatible almost complex structure defined on an open set, it extends to one on $Z$.
\end{proof}

We make the following definition, then prove an end Floer nonvanishing result, which is a modification of Gadgil's Theorem~1.6 in \cite{gadgil}.

\begin{definition}\label{def:steincassonhandle}
A \textit{Stein admissible generalized Casson Handle} is a generalized Casson handle such that each of its defining cobordisms $n_k$ admits a handle decomposition as Stein handle attachments to the lower boundary, and each cobordism $n_kn_{k+1}...n_\ell$ is admissible in the sense of Gadgil.
\end{definition}

\begin{theorem}\label{thm:exotic}
    Let $\cR$ be a smooth manifold made from a slice disk complement for a knot $K$ of genus $g\ge 2$, union a Stein admissible generalized Casson handle attached along the meridian of the disk with 0 framing, and deleting the remaining boundary. Then $\cR$ is an exotic $\R^4$, and $\cR\#_\infty \overline{\bbC P^2} $ is not diffeomorphic to $\R^4\#_\infty \overline{\bbC P^2} $.
\end{theorem}

\begin{proof}
     Notice $\cR$ has a compact exhaustion $\{C_n\}$ where we define $C_0$ to be the disk complement for $K$, and $C_i$ is constructed by attaching the finite truncation $n_1n_2...n_i$ of the generalized Casson handle to $C_0$ along a 0-framed meridian of the deleted disk. To each $C_i$ we add a product neighborhood to the boundary, ensuring their union is $\cR$. Let $W_{ij} := C_j - \text{int}(C_i)$ and let $Y_i := \partial C_i$. Then $\{C_i\}$ is an admissible exhaustion as in Definition~\ref{def:admissible}, 
     hence can be used to compute the end Floer homology:
    \[\HE(\cR,\s) = \varinjlim \HFred(Y_i, \s|_{Y_i}).\]
    It suffices to show this group is nonzero for some choice of asymptotic spin$^c$ structure. The first 3-manifold in this direct system is $Y_0 = S^3_0(K)$; 
    by Gabai's Theorem \cite{gabaifoliations} $Y_0$ admits a taut foliation, and by work of Eliashberg and Thurston \cite{confoliations, eliashbergfilling}, the product $Y_0\times I$ admits a symplectic strucure weakly filling both boundary components (the same as in Lemma \ref{lem:symplectic}). Note that $Y_0 \times I$ union the Stein admissible generalized Casson handle is a symplectic neighborhood of the end of $\cR$. The compatible spin$^c$ structure $\t'$ on this neighborhood will be our asymptotic spin$^c$ structure, that is, we will show $HE(\cR,\t')\ne 0$.
    
    Adjoining the Stein cobordism $W_{0j}$ to the upper boundary of $Y_0\times I$ yields a symplectic manifold with two convex ends. Work of Eliashberg \cite{eliashbergfilling}, and Etnyre \cite{etnyrefillings}, on symplectic caps shows there exist closed symplectic manifolds
    \[Z_j = X_1' \cup( Y_0\times I )\cup W_{0j} \cup X_j'\]
    where $X_1'$ and $X_j'$ have $b_2^+ \ge 2$. Note $X_1'$ is the same in each of the $Z_j$.

    Since $b_2^+(X_1')>1$ we may choose an admissible cut $P$ in $X_1'$. Let $U,V$ denote components of $X_1'\setminus P$ where $ V$ is a cobordism from $P$ to $Y_0$. Note that for any spin$^c$ structure $\s$ on $Z_j$, one can use $P$ to compute the mixed invariant $\Phi_{Z_j,\s}$. That is, given a fixed $\s$ on $Z_j$, there exists $\theta \in \HFp(P,\s|_P)$ with 
    \[\Phi_{Z_j,\s} = F^+_{(X_j'\setminus B^4)\circ W_{0j}\circ V, \s}(\theta).\]
    Fix the canonical spin$^c$ structure $\t$ on $Z_j$ (suppressing the subscript $j$). Then, writing the composition law for $(X_j'\setminus B^4)\circ W_{0j}\circ V$, each with their own restrictions of $\t$, we find
    \[\sum_{\eta \in H^1(Y_0;\mathbb{Z}) } F^+_{(X_j'\setminus B^4)\circ W_{0j}\circ V, \t + \delta\eta} =F^+_{(X_j'\setminus B^4) ,\t} \circ F^+_{W_{0j}, \t} \circ F^+_{V, \t} \]
    where $\delta$ is the Mayer Vietoris coboundary map, and the set of spin$^c$ structures in the sum simplifies due to admissibility of $W_{0j}$. Since all spin$^c$ structures in the sum restrict to $\t|_{P}$ on $P$, we may choose a single element $\theta \in \HFp(P,\t|_P)$ such that 
\[\sum_{\eta \in H^1(Y_0;\mathbb{Z}) } \Phi_{Z_j, \t + \delta\eta} = \sum_{\eta \in H^1(Y_0;\mathbb{Z}) } F^+_{(X_j'\setminus B^4)\circ W_{1,j}\circ V, \t + \delta\eta}(\theta) =F^+_{(X_j'\setminus B^4) ,\t} \circ F^+_{W_{1,j}, \t} \circ F^+_{V, \t} (\theta) .\]
We would like to show the right hand side of the above is nonzero. By the symplectic nonvanishing theorem \cite{ossymplectic} we know that $\Phi_{Z_j,\t}$, as an element of $\HFp(S^3)$ is nonzero in grading 0. But a priori, this term could cancel with others in the sum. We use gradings to show this does not happen: in particular, we will show that as $\eta$ varies through $H^1(Y_0,\mathbb{Z})$, the degrees 
\[d(\t + \delta\eta) = \frac{c_1(\t  +\delta\eta)^2 - 2\chi(Z_j) - 3\sigma(Z_j)}{4}\] are all distinct. Then it follows by Section 9 of \cite{OSfourmanifolds} that for $\eta\ne 0$, the element $\Phi_{Z_j,\t + \delta\eta}$ is zero in grading zero, hence the sum above is nonzero. It suffices to show the quantities $c_1(\t + \delta \eta)^2$ are all distinct. Since $H^1(Y_0;\mathbb{Z}) = \mathbb{Z}$ we may write $\delta(\eta) = n \delta(1)$ for some $n \in \bbZ$, then
\begin{align*}
    c_1(\t + \delta \eta)^2 = (c_1(\t) + 2n\delta(1))^2
    = c_1(\t)^2 + 2n^2 \delta(1)^2 + 4nc_1(\t) \cup \delta(1).
\end{align*}
Since $\delta$ is Poincar\'e dual to the inclusion induced map $H_2(Y_0)\to H_2(Z_j)$, it follows that $\delta(1)^2 = 0$. To demonstrate $c_1(\t) \cup \delta(1) \ne 0$, note this is the same as
$\langle c_1(\t), i_*[\Sigma]\rangle = \langle c_1(\omega) , i_*[\Sigma]\rangle,$
where $\Sigma \subset S^3_0(K)$ is obtained by capping a minimal-genus Seifert surface for $K$ with a disk \cite{gabaifoliations}. By Lemma~\ref{lem:symplectic} above and the fact that $K$ has genus at least 2, we have
\[\langle c_1(\t), i_*[\Sigma]\rangle = \langle i^* c_1(\t), [\Sigma]\rangle = \langle c_1(T\Sigma), [\Sigma]\rangle = \chi(\Sigma) = 2-2g  < 0. \]
Hence
$c_1(\t + n\delta(1))^2 = c_1(\t)^2 + 8n(1-g),$ which are all distinct as $n$ varies. 
This verifies $F^+_{(X_j'\setminus B^4) ,\t} \circ F^+_{W_{1,j}, \t} \circ F^+_{V, \t} (\theta) \in \HFp(S^3)$ is nonzero. Since the cobordism $X_j'\setminus B^4$ has $b_2^+ >1$ it follows that $F^+_{W_{1,j}, \t} \circ F^+_{V, \t} (\theta)$ projects to a nonzero element of $\HFred(Y_j, \t|_{Y_j})$ for each $j$. Thus, the images of $F^+_{V, \t} (\theta) + \HFinf(Y_0)$ under the cobordism maps $F^+_{W_{0j}}$ are all nonzero, and since $\t|_{Y_j} = \t'|_{Y_j}$ for all $j$, it follows that $\HE(\cR,\t')\ne 0$.

To verify the claim about blowups, note that attaching a 2-handle along a $-1$-framed unknot, unlinked from a given surgery diagram, is an admissible cobordism in the sense of Gadgil (as it does not affect $H^1$ of the 3- or 4-manifold). Thus, arranging the cobordisms $W_{i,i+1}$ to include finitely many blowups each yields an admissible exhaustion of $\cR\#_\infty \overline{\bbC P^2}$. The same steps as before, but using the blown-up $Z_j$, show nonvanishing end Floer homology of $\cR\#_\infty \overline{\bbC P^2}$. On the other hand, note that $\R^4\#_\infty \overline{\bbC P^2}$ has an admissible exhaustion with smooth $S^3$'s running off the end, hence has vanishing end Floer homology.
\end{proof}

\subsection{Distinguishing exotic planes}\label{sec:detecting}
We prove Theorem~\ref{theorem:exoticR4}, which states that $\cR$ is a Stein exotic $\R^4$ that embeds in $\R^4$, and $\cR$, $\overline{\cR}$ and their end sum $\cR\natural \overline{\cR}$ are three distinct exotic $\R^4$'s. Distinguishing the three exotic $\R^4$'s is done using the stable diffeomorphism technique from Bi\v zaca-Gompf \cite{bizacagompf}, which we show is compatible with end Floer homology.
\begin{proposition}\label{prop:exoticR4s}
   The manifolds $\cR$, $\overline{\cR}$ and their end sum $\cR\natural \overline{\cR}$ are three distinct, nonstandard, exotic $\R^4$'s, hence their branch sets are distinct exotic planes.    
\end{proposition}
   \begin{figure}[ht]
\centering
            \begin{overpic}[width=\textwidth, 
 unit=0.5mm, tics=10]{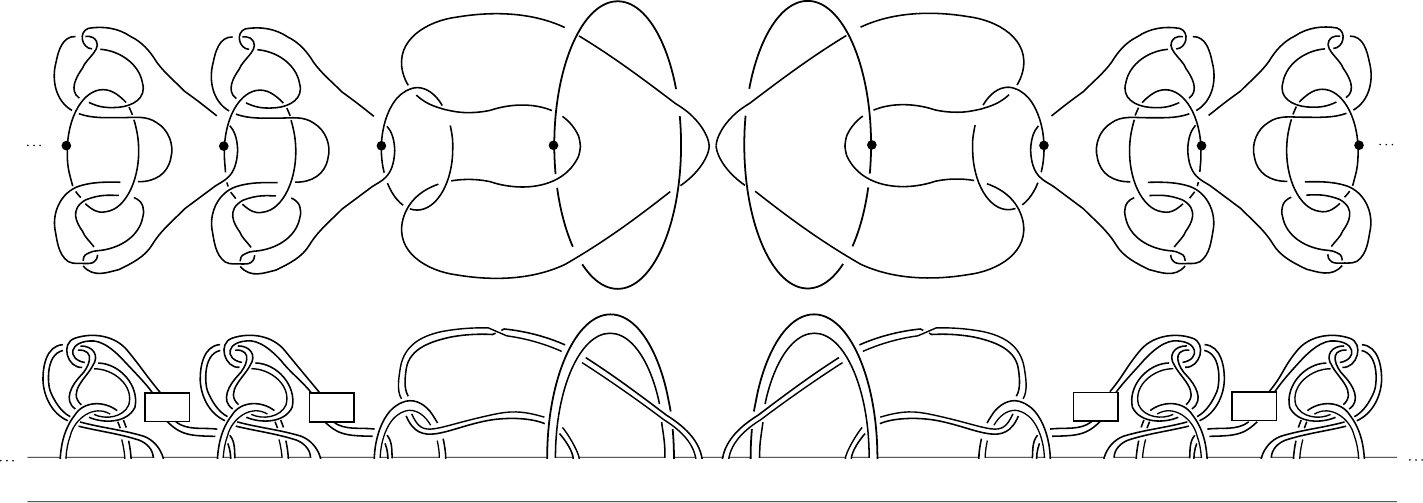}
        \put(33,35){$1$}
        \put(65,35){$-1$}
        \put(75.5,6){$-2$}
        \put(86.6,6){$-2$}
        \put(11.3,6){$2$}
        \put(22.7,6){$2$}
        \put(80,33.5){$0$}
        \put(91,33.5){$0$}
        \put(8,33.5){$0$}
        \put(19,33.5){$0$}
    \end{overpic}
        \caption{Top: the end sum $\cR\natural \overline{\cR}$ with visible symmetry. Bottom: the branch set in $\R^4$.}
        \label{fig:plane2}
\end{figure}
\begin{proof}
These manifolds are homeomorphic to $\R^4$ by Theorem~\ref{thm:open2handles}.
     Note, $TH$ is Stein by Theorem~\ref{thm:stein} and has an admissible exhaustion by \cite[Lemma 2.3]{gadgil}. By this and Proposition~\ref{prop:diskcomplements}, we see $\cR$ is a disk complement for a genus 2 knot union a Stein generalized Casson handle. Thus, Theorem~\ref{thm:exotic} shows $\cR$ is an exotic $\R^4$, and that $\cR$ is $\overline{\bbC P^2}$-stably nonstandard. We will show $\cR$ is $\bbC P^2$-stably standard. This property is not shared by $\overline{\cR}$, since $\overline{\cR} \#_\infty \bbC P^2 \cong \overline{\cR \#_\infty \overline{\bbC P^2} }$ has nonvanishing end Floer homology with its opposite orientation. Similarly we will show $\cR\natural\overline{\cR}$ is not $\bbC P^2$ or $\overline{\bbC P^2}$ stably standard, hence all three are distinct and nonstandard. Note that $\cR$ admits a compact exhaustion by manifolds $C_k$, which are the disk complement $X_{-1}$, union the first $k$ copies of $Tk$, union a collar of the boundary. Adding infinitely many blowups we may assume $C_k \#_{2k}\bbC P^2 \subset \text{int}(C_{k+1} \#_{2(k+1)}\bbC P^2)$ for each $k$. Using the two $\bbC P^2$'s of the $k+1$th stage, we may undo the top stage clasps of $C_k$ preserving the 0-framing.
But since $C_k$ union a zero framed 2-handle attached to its tip circle is $B^4$, it follows that for any $k$ the embedding\[C_k \#_{2k}\bbC P^2 \into \text{int}(C_{k+1} \#_{2(k+1)}\bbC P^2)\]
factors through an embedding of $B^4\#_{2k}\bbC P^2$. Thus, $\cR$ admits a compact exhaustion by smooth nested $B^4\#_{2k}\bbC P^2$'s (where two new blowups are added at each stage). Since the complement of a smooth ball in a smooth ball is a standard annulus, it is now straightforward to construct a diffeomorphism from $\cR\#_\infty \bbC P^2$ to $\R^4_{std}\#_\infty \bbC P^2$.
Similarly, by adding infinitely many negative blowups to  $\cR\natural\overline{\cR}$ we may ambiently isotop them to be on one side of the end-summing $\R^3$, i.e. on the side of $\overline{\cR}$. The proof above shows $\overline{\cR}\#_\infty \overline{\bbC P^2} \cong \R^4_{std}\#_\infty \overline{\bbC P^2}$  and it follows that $\cR\natural \overline{\cR}\#_\infty \overline{\bbC P^2}$ is diffeomorphic to ${\cR} \#_\infty \overline{\bbC P^2}$, which we saw is not standard. The same holds for positive blowups. By drawing Kirby diagrams for $\overline{\cR}$ and $\cR\natural \overline{\cR}$ where their infinite chains lie along the horizontal axis, it is clear that in either case, rotating by $\pi$ about this axis gives a double branched covering over a plane in $\R^4$. We draw part of the plane for $\cR\natural \overline{\cR}$ in Figure \ref{fig:plane2}. By Teng \cite{endkhovanov} it is straightforward to see that the planes are topologically standard.
\end{proof}
Recall that an exotic $\bbR^4$ is called \textit{small} if all of its compact, codimension-0 submanifolds embed smoothly in $\bbR^4_{std}$, and is called \textit{large} otherwise. It is an open question whether every small exotic $\bbR^4$ embeds in $\bbR^4_{std}$.
\begin{proof}[Proof of Theorem~\ref{theorem:exoticR4}]
We saw $\cR$ is Stein in Theorem~\ref{thm:stein}. Smallness of $\cR$ follows from Theorem~\ref{thm:abstractcasson}, which shows $TH$ embeds in a standard 2-handle rel boundary. Exoticness is shown in Proposition~\ref{prop:exoticR4s}.
\end{proof}

\begin{remark}\label{remark:endfoer}
    The same technique applies to any generalized ribbon $\R^4$, made with a Stein admissible generalized Casson handle with only positive clasps, to distinguish between its reverse orientation and their end sum. In particular, this holds if we combine the disk complement $X_n$ (from Section~\ref{sec:handle}) with $TH$. 
    Note for all $n$, the knot Floer homology of $K_n$ is the same by Proposition \ref{prop:diskcomplements}, and it is likely all the $n$th-stage $HF^+$ modules will coincide. Thus, the present method (and the more refined technique of the Hom, Lidman, and the author \cite{elihomlidman}) seem unlikely to be able to distinguish these.
\end{remark}

\section{Four-stage disks and topologically slice links }\label{sec:disk} 

Given a 4-manifold with boundary $(M,\partial M)$ and a framed circle $C \subset \partial M$, \textit{disk embedding theory} gives techniques for finding topologically flat disks $(D,\partial D) \into (M,\partial M)$. Cha and Powell use such a technique to find topologically flat core disks in the first four stages of any Casson handle (i.e. in any height 4 Casson tower) \cite[Theorem~A]{chapowell}. We adapt this technique to the \textit{banded Casson towers} of Section~\ref{sec:THTriangular}, which include $(TH, \partial_-TH)$. We review the relevant definitions, following notation of \cite{chapowell}. For more details on disk embedding theory see \cite{diskembedding, chapowell, freedmanquinn}. 

In this section a \textit{grope of height $n$} denotes a type of 2-complex with circle boundary, constructed inductively as follows. Start with a genus $g$ surface with a single boundary component (this is a grope of height 1); to each element of a symplectic basis of the first homology of the first stage surface, attach a height 1 grope by a homeomorphism from its boundary circle to the basis curve, building a second stage of surfaces and yielding a height 2 grope, and so on. For a height $n$ grope $G_n$, we denote stages $p$ through $q$ inclusive by $G_{p-q}$. A \textit{capped grope of height $n$} is obtained from a grope of height $n$ by attaching disks to its top stage symplectic bases, and will be denoted $G_n^c$. 
In \cite{chapowell}, a standard model embedding of a capped grope in $\R^3$ is given. Embedding this standard model into $\R^4 = \R^3\times \R$ and taking a 4-dimensional tubular neighborhood yields a \textit{model framed (capped) grope}, denoted $\widetilde{G_n^c}$; the attaching region $\partial_- \widetilde{G_n^c}$ is defined to be the solid torus obtained by thickening the base circle of the first stage. The normal bundle of any surface or disk of a model capped grope in $\R^4$ is framed by taking a 1-dimensional framing in $\R^3$, times the $\R$ factor in $\R^3\times \R$. The attaching circle and all symplectic bases obtain the induced framing, and we use this framing throughout. Finally, a \textit{properly immersed (capped) grope} in a 4-manifold $(M,\partial M)$ is obtained by adding plumbings and self-plumbings to the caps in a model framed capped grope $\widetilde{G_n^c}$, and taking the image of the embedding $(\widetilde{G_n^c},\partial_-\widetilde{G_n^c}) \hookrightarrow (M,\partial M)$. We will also denote properly immersed capped gropes by $({G_n^c},\partial_-{G_n^c})$.

 We will use the following theorem, adapted from \cite{chapowell}, to find core disks in banded Casson towers.

\begin{theorem}[Cha-Powell, Theorem 3.4]\label{thm:chapowelldisk}
    Let $(M,\partial M)$ be a smooth 4-manifold with boundary and let $n \in \bbN$ with $n\ge 2$. Suppose $(G^c_n,\partial_-G^c_n)\to (M,\partial M)$ is a properly immersed capped grope of height $n$ and $\nu G_n^c$ is a further thickening of $G_n^c$. If the image of the inclusion induced map
    \[\pi_1(\nu G_n^c \setminus G_{1-1},*)\to \pi_1(M\setminus G_{1-1},*)\]
    is of subexponential growth for all choice of basepoints, then there is a topologically flatly embedded disc in $M$ with the same framed boundary as $G_n^c$.
\end{theorem}

In order to apply Theorem~\ref{thm:chapowelldisk}, we set up two lemmas, which are modifications of lemmas 4.2 and 4.3 of~\cite{chapowell}. In the following, $T$ denotes a banded Casson tower, and we write $T_{1-n}$ for the first $n$ stages of $T$.

\begin{lemma}\label{lem:grope}
Let $T$ be a banded Casson tower. Then $T_{1-3}$ contains a properly immersed capped grope of height $2$, $G_2^c$, whose framed attaching circle  coincides with the framed attaching circle $C$ of $T_{1-3}$.
\end{lemma}

  \begin{figure}[htb]
    \centering
        \begin{subfigure}[b]{0.3\textwidth}
            \begin{overpic}[width=\textwidth, 
 unit=1mm, tics=5]{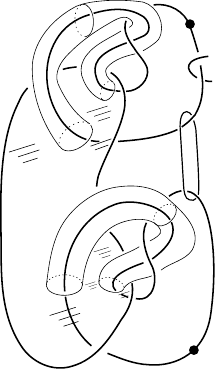}
            \put(46, 55){$0$}
 \end{overpic}
            \caption{}
            \label{fig:gropestage1}
        \end{subfigure}\hspace{8mm}
        \begin{subfigure}[b]{0.31\textwidth}
            \begin{overpic}[width=\textwidth, 
 unit=1mm, tics=5]{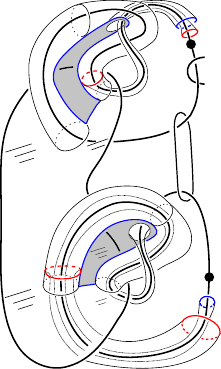}
            \put(45, 55){$0$}
 \end{overpic}
            \caption{}
            \label{fig:gropestage2}
        \end{subfigure}\hspace{5mm}
        \begin{subfigure}[b]{0.24\textwidth}
            \begin{overpic}[width=\textwidth, 
 unit=1mm, tics=5]{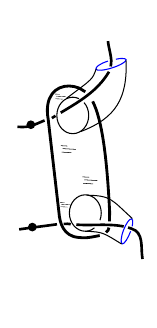}
             \put(13, 48){$0$}
 \end{overpic}
            \caption{}
            \label{fig:gropestage3}
        \end{subfigure}
       \caption{(a): First stage grope $\Sigma$ in a Kirby diagram of $Tk$. (b): Symplectic basis of $\Sigma$ and extending annuli. (c): Surface used to `tube across the 2-handle', drawn near the 2-handle of (b).
       }
        \label{fig:grope}
    \end{figure}

\begin{proof}
    We first consider the case $T=TH$, and then explain how to adapt the proof to the general case. Recall $TH$ is a union of $Tk$ units.
    
    We construct $G_2^c$ following Ray's construction of a grope in a Casson tower \cite{Ray}, which readily generalizes to $TH$. We first find the 2-complex serving as the spine of the grope, then take a tubular neighborhood. The first stage surface $\Sigma$ is the genus 2 surface indicated in Figure~\ref{fig:gropestage1}, with boundary equal to the attaching circle of $Tk$. Note, the tip circle of $Tk$ is indicated as a meridian to the upper 1-handle, and this is where the second stage $Tk$ is attached. Figure~\ref{fig:gropestage2} describes four embedded annuli, interiors disjoint from $\Sigma$, extending from a symplectic basis for $H_1(\Sigma)$. Lengthen the lower two annuli by attaching two disjoint parallel copies of the surface indicated in Figure~\ref{fig:gropestage3}, union two disjoint cores of the 2-handle. We call this `tubing across a 2-handle'. 
    
    We have arranged that all four loose boundary components of the annuli are parallel copies of the tip circle of $Tk$. Thus, they bound disjoint genus 2 surfaces obtained from the second stage copy of $Tk$: these are parallel copies of punctured, second-stage copies of the surface described in Figure~\ref{fig:gropestage1}, union cores of the second-stage 2-handle $h$. The resulting complex $G_2$ is embedded, and constitutes the spine of the surface stages of $G_2^c$. 
    We obtain caps for $G_2$ by tubing on copies of core disks for the third stage $Tk$ as follows. Extend 16 annuli from the top stage symplectic bases for $G_2$ as in Figure~\ref{fig:gropestage2}. These will intersect each other in disjoint circles and arcs, but we can always push one sheet into $B^4$ to obtain embedded annuli. Tubing on top stage immersed core disks of the third $Tk$, as we defined in Section~\ref{sec:THTriangular} when we defined $Tk$, will provide the wrong framing: this is because the attaching framing for $C\subset Tk$ and the framing coming from the immersed core disk $D\subset Tk$ differ by four twists (see Section \ref{sec:handle}). To fix this, locally add two negative double points to the top stage disk, obtaining a new disk $D'\subset Tk$ whose induced framing on the attaching circle coincides with the canonical framing. Now, tubing copies of $D'$ on to $G_2$ gives a height 2 capped grope. Pushing it into the interior and taking a tubular neighborhood gives a height 2 framed embedded grope with framed attaching circle $C$ as desired. 

    For general $T$, recall the banded plumbed 2-handle $T_{1-1}$ admits a Kirby diagram as the standard diagram for a plumbed 2-handle $k$, with several 0-framed 2-handles attached, analogous to Figure~\ref{fig:gropestage1}. An example of a banded plumbed handle is shown in Figure~\ref{fig:triangular}. Thus, we may define a first stage surface $\Sigma \subset T_{1-1}$ as a multiply punctured disk $\widetilde\Sigma$ bounded by the attaching circle, with one tube added for each double point of $k$. This surface might intersect the 2-handles (and their defining bands) but by shrinking its tubes we can assume the intersections occur in $\text{int}(\widetilde\Sigma)$. Thus, the intersections can be eliminated by pushing small $D^2$ neighborhoods in $\widetilde\Sigma$ into $B^4$. Extending annuli from a symplectic basis for $\Sigma$, and tubing across the 0-framed triangular 2-handles as before, we may arrange that the loose boundary components of these annuli are parallel copies of tip circles for $T_{1-1}$ (recall, the tip circles are a subcollection of meridans to the dotted circles of $k$.) Moreover, these annuli are disjoint, since we eliminated intersections between $\widetilde\Sigma$ and the defining bands for the 2-handles. Completing the second stage and caps proceeds as before, with various banded plumbed 2-handles used instead of $Tk$, and analogues of the surface we defined in $T_{1-1}$. Note that intersections produced by tubing across 2-handles (e.g. if a band passes through a band) can be eliminated by pushing one sheet into $B^4$. 
\end{proof}
    
\begin{lemma}\label{lem:pi1}
 Let $T$ be a banded plumbed 2-handle, and let $\Sigma \subset T$ denote the first stage surface of Lemma~\ref{lem:grope}. Then $\pi_1(T\setminus \Sigma) \cong \langle \mu,\mu_1,...,\mu_k| [\mu,\mu_i], i=1,...,k \rangle$ where $\mu$ is a meridian of $\Sigma$ and the $\mu_i$ are tip circles to $T$.
\end{lemma}

\begin{proof}
Let $k$ be the underlying plumbed 2-handle to $T$. The surface $\Sigma\subset T$ of Lemma~\ref{lem:grope} is contained in $k$, and coincides with the first stage surface $\Sigma'$ of the grope constructed by Ray in $k$ \cite{Ray}. Lemma 4.3 of \cite{chapowell} shows $\pi_1(k\setminus \Sigma') \cong \langle \mu,a_1,...,a_n| [\mu,a_i], i=1,...,n \rangle$ where $\mu$ is a meridian to $\Sigma'$ and $a_i$ are the double point loops of $k$. Then $T\setminus \Sigma = (k\setminus \Sigma')$ union 2-handles, which identify pairs of $a_i$ until the tip circles remain.
\end{proof}

\begin{theorem}\label{thm:chapowell}
    Any banded Casson tower $T$ contains a topologically flat core disk in its first four stages, bounded by its framed attaching circle $C$. In particular this holds for $TH$.
\end{theorem}

\begin{proof}
     Lemma \ref{lem:grope} gives a properly embedded immersed capped grope $G_2^c$ in $T_{1-3}$ bounded by the framed attaching circle $C$. Write $\Sigma$ for its first stage surface. Since $\pi_1(T_{1-3})$ is generated by the top stage tip circles, the inclusion map $\pi_1(T_{1-3}) \to \pi_1(T_{1-4})$ is trivial. By Lemma \ref{lem:pi1} it follows that the map
     $\pi_1(T_{1-3}\setminus \Sigma) \to \pi_1(T_{1-4}\setminus \Sigma)$
     induced by inclusion has image $\bbZ$ generated by a meridian of $\Sigma$. Since $\nu G_2^c \subset T_{1-3}$ and $\pi_1(\nu G_2^c\setminus \Sigma)$ is generated by the caps' double point loops (which are all homotopic to tip circles of $T_{1-3}$) together with a meridian to $\Sigma$, it follows that the inclusion induced map
     $\pi_1(\nu G_2^c\setminus \Sigma) \to \pi_1(T_{1-4}\setminus \Sigma)$ has image the same $\bbZ$. Theorem \ref{thm:chapowelldisk} now gives the desired disk.
\end{proof}

Using Theorem~\ref{thm:chapowell}, we prove Theorem~\ref{thm:TH}, and slice the links of Theorem~\ref{thm:links} as in the proof of \cite[Theorem F]{chapowell}.

\begin{proof}[Proof of Theorem \ref{thm:TH}]
    The proof follows from  Propositions~\ref{prop:THhomeo} and \ref{prop:notaplumbedhandle}, and Theorem~\ref{thm:chapowell}.
\end{proof}

\begin{proof}[Proof of Theorem \ref{thm:links}]
Given a height 4 banded Casson tower $T_{1-4}$, draw a Kirby diagram and cancel all possible 1-2-handle pairs. This gives a diagram of $B^4$ where we see the attaching circle of $T_{1-4}$ linking band-sums of the top stage dotted circles, in a 4-fold banded ramified Whitehead pattern applied to one component of a Hopf link. Any such pattern can be obtained by choosing an appropriate $T$ (e.g. compare Figures~\ref{fig:link} and \ref{fig:triangular}). Note, this Kirby diagram determines an embedding of $T_{1-4}$ into $B^4$: it is exactly the complement of the slice disks corresponding to the band-summed dotted circles. Since $T_{1-4}$ has a flat core disk, the desired link is topologically slice.  
\end{proof}

\bibliographystyle{abbrv}
\bibliography{mybib}

@article{freedmandemichelis,
author = {Stefano De Michelis and Michael H. Freedman},
title = {{Uncountably many exotic $\mathbf{R}^4$'s in standard 4-space}},
volume = {35},
journal = {Journal of Differential Geometry},
number = {1},
publisher = {Lehigh University},
pages = {219 -- 254},
year = {1992},
doi = {10.4310/jdg/1214447810},
URL = {https://doi.org/10.4310/jdg/1214447810}
}

@article{endsofmaps,
  title     = {Ends of maps. {III}. {D}imensions 4 and 5},
  author    = {Quinn, Frank},
  journal   = {Journal of Differential Geometry},
  volume    = {17},
  number    = {3},
  pages     = {503--521},
  year      = {1982},
  publisher = {Lehigh University},
  doi       = {10.4310/jdg/1214437139},
  url       = {https://projecteuclid.org/journals/journal-of-differential-geometry/volume-17/issue-3/Ends-of-maps-III-Dimensions-4-and-5/10.4310/jdg/1214437139.full}
}

@book{diskembedding,
    author = {Behrens, Stefan and Kalmar, Boldizsar and Kim, Min Hoon and Powell, Mark and Ray, Arunima},
    title = {The Disc Embedding Theorem},
    publisher = {Oxford University Press},
    year = {2021},
    month = {07},
    isbn = {9780198841319},
    doi = {10.1093/oso/9780198841319.001.0001},
    url = {https://doi.org/10.1093/oso/9780198841319.001.0001},
}

@article{gadgil,
title = {Open manifolds, {O}zsv\'ath--{S}zab\'o invariants and exotic $\mathbb{R}^4$'s},
journal = {Expositiones Mathematicae},
volume = {28},
number = {3},
pages = {254-261},
year = {2010},
issn = {0723-0869},
doi = {https://doi.org/10.1016/j.exmath.2009.09.002},
url = {https://www.sciencedirect.com/science/article/pii/S0723086909000577},
author = {Siddhartha Gadgil}
}

@article{bizacagompf,
author = {{\v Z}arko Bi{\v z}aca and Robert E. Gompf},
title = {{Elliptic surfaces and some simple exotic ${\bf R}\sp 4$'s}},
volume = {43},
journal = {Journal of Differential Geometry},
number = {3},
publisher = {Lehigh University},
pages = {458 -- 504},
year = {1996},
doi = {10.4310/jdg/1214458322},
URL = {https://doi.org/10.4310/jdg/1214458322}
}

@article{freedman,
author = {Michael Hartley Freedman},
title = {{The topology of four-dimensional manifolds}},
volume = {17},
journal = {Journal of Differential Geometry},
number = {3},
publisher = {Lehigh University},
pages = {357 -- 453},
year = {1982},
doi = {10.4310/jdg/1214437136},
URL = {https://doi.org/10.4310/jdg/1214437136}
}

@book {gompfstipsicz,
    AUTHOR = {Gompf, Robert E. and Stipsicz, Andr\'as I.},
     TITLE = {{$4$}-manifolds and {K}irby calculus},
    SERIES = {Graduate Studies in Mathematics},
    VOLUME = {20},
 PUBLISHER = {American Mathematical Society, Providence, RI},
      YEAR = {1999},
     PAGES = {xvi+558},
      ISBN = {0-8218-0994-6},
   MRCLASS = {57N13 (14J80 32Q55 57-02 57R17 57R57 57R65)},
  MRNUMBER = {1707327},
MRREVIEWER = {Nikolai\ N.\ Saveliev},
       DOI = {10.1090/gsm/020},
       URL = {https://doi.org/10.1090/gsm/020},
}

@article{gompfinfinite,
author = {Robert E. Gompf},
title = {{An infinite set of exotic $\mathbf{R}^4$'s}},
volume = {21},
journal = {Journal of Differential Geometry},
number = {2},
publisher = {Lehigh University},
pages = {283 -- 300},
year = {1985},
doi = {10.4310/jdg/1214439566},
URL = {https://doi.org/10.4310/jdg/1214439566}
}

@misc{AIMlist,
  title        = {Smooth concordance classes of topologically slice knots},
  author       = {{American Institute of Mathematics}},
  howpublished = {\url{http://aimpl.org/concordsliceknot/5/}},
  note         = {AIM Problem List, Problem 5.4: ``Explicitly describe some non-smooth topologically slice disc.'' Accessed 9 June 2026},
}

@article {ManolescuMarengon,
    AUTHOR = {Manolescu, Ciprian and Marengon, Marco and Sarkar, Sucharit
              and Willis, Michael},
     TITLE = {A generalization of {R}asmussen's invariant, with applications
              to surfaces in some four-manifolds},
   JOURNAL = {Duke Math. J.},
  FJOURNAL = {Duke Mathematical Journal},
    VOLUME = {172},
      YEAR = {2023},
    NUMBER = {2},
     PAGES = {231--311},
      ISSN = {0012-7094,1547-7398},
   MRCLASS = {57K18 (57K40)},
  MRNUMBER = {4541332},
MRREVIEWER = {William\ Rushworth},
       DOI = {10.1215/00127094-2022-0039},
       URL = {https://doi.org/10.1215/00127094-2022-0039},
}

@article{endkhovanov,
      title={End {K}hovanov homology and exotic {L}agrangian planes}, 
      author={Yikai Teng},
      year={2025},
      eprint={2510.01151},
      journal={arXiv: 2510.01151},
      primaryClass={math.GT},
      url={https://arxiv.org/abs/2510.01151}, 
}

@article {akbulutkirby,
    AUTHOR = {Akbulut, Selman and Kirby, Robion},
     TITLE = {Branched covers of surfaces in {$4$}-manifolds},
   JOURNAL = {Math. Ann.},
  FJOURNAL = {Mathematische Annalen},
    VOLUME = {252},
      YEAR = {1979/80},
    NUMBER = {2},
     PAGES = {111--131},
      ISSN = {0025-5831,1432-1807},
   MRCLASS = {57M12 (14J99 57N15)},
  MRNUMBER = {593626},
MRREVIEWER = {R.\ Mandelbaum},
       DOI = {10.1007/BF01420118},
       URL = {https://doi.org/10.1007/BF01420118},
}

@article {chapowell,
    AUTHOR = {Cha, Jae Choon and Powell, Mark},
     TITLE = {Casson towers and slice links},
   JOURNAL = {Invent. Math.},
  FJOURNAL = {Inventiones Mathematicae},
    VOLUME = {205},
      YEAR = {2016},
    NUMBER = {2},
     PAGES = {413--457},
      ISSN = {0020-9910,1432-1297},
   MRCLASS = {57N13 (57M25 57N70)},
  MRNUMBER = {3529119},
MRREVIEWER = {Vyacheslav\ S.\ Krushkal},
       DOI = {10.1007/s00222-015-0639-z},
       URL = {https://doi.org/10.1007/s00222-015-0639-z},
}

@article{elihomlidman,
      title={Distinguishing exotic $\mathbb{R}^4$'s with {H}eegaard {F}loer homology}, 
      author={Sean Eli and Jennifer Hom and Tye Lidman},
      year={2026},
      eprint={2601.08767},
      journal={arXiv: 2601.08767},
      primaryClass={math.GT},
      url={https://arxiv.org/abs/2601.08767}, 
}

@misc{knotjob,
  title = {{K}not{J}ob},
  author = {Dirk Sch{\"u}tz},
  howpublished = {Department of Mathematical Sciences, Durham University},
  url = {https://www.maths.dur.ac.uk/users/dirk.schuetz/knotjob.html}
}

@article {eliashbergfilling,
    AUTHOR = {Eliashberg, Yakov},
     TITLE = {A few remarks about symplectic filling},
   JOURNAL = {Geom. Topol.},
  FJOURNAL = {Geometry and Topology},
    VOLUME = {8},
      YEAR = {2004},
     PAGES = {277--293},
      ISSN = {1465-3060,1364-0380},
   MRCLASS = {57R17 (53D35 57R57)},
  MRNUMBER = {2023279},
MRREVIEWER = {A.\ Stipsicz},
       DOI = {10.2140/gt.2004.8.277},
       URL = {https://doi.org/10.2140/gt.2004.8.277},
}

@article{Ray,
   title={Casson towers and filtrations of the smooth knot concordance group},
   volume={15},
   ISSN={1472-2747},
   url={http://dx.doi.org/10.2140/agt.2015.15.1119},
   DOI={10.2140/agt.2015.15.1119},
   number={2},
   journal={Algebraic \& Geometric Topology},
   publisher={Mathematical Sciences Publishers},
   author={Ray, Arunima},
   year={2015},
   month=Apr, pages={1119–1159} }

@article {OSgenus,
    AUTHOR = {Ozsv\'ath, Peter and Szab\'o, Zolt\'an},
     TITLE = {Holomorphic disks and genus bounds},
   JOURNAL = {Geom. Topol.},
  FJOURNAL = {Geometry and Topology},
    VOLUME = {8},
      YEAR = {2004},
     PAGES = {311--334},
      ISSN = {1465-3060,1364-0380},
   MRCLASS = {57M27 (53D35 57N10 57R58)},
  MRNUMBER = {2023281},
MRREVIEWER = {Jacob\ Andrew\ Rasmussen},
       DOI = {10.2140/gt.2004.8.311},
       URL = {https://doi.org/10.2140/gt.2004.8.311},
}

@book {confoliations,
    AUTHOR = {Eliashberg, Yakov M. and Thurston, William P.},
     TITLE = {Confoliations},
    SERIES = {University Lecture Series},
    VOLUME = {13},
 PUBLISHER = {American Mathematical Society, Providence, RI},
      YEAR = {1998},
     PAGES = {x+66},
      ISBN = {0-8218-0776-5},
   MRCLASS = {53C15 (57N10 57R30 58F05)},
  MRNUMBER = {1483314},
MRREVIEWER = {Hansj\"org\ Geiges},
       DOI = {10.1090/ulect/013},
       URL = {https://doi.org/10.1090/ulect/013},
}

@article {gabaifoliations,
    AUTHOR = {Gabai, David},
     TITLE = {Foliations and the topology of {$3$}-manifolds},
   JOURNAL = {Bull. Amer. Math. Soc. (N.S.)},
  FJOURNAL = {American Mathematical Society. Bulletin. New Series},
    VOLUME = {8},
      YEAR = {1983},
    NUMBER = {1},
     PAGES = {77--80},
      ISSN = {0273-0979,1088-9485},
   MRCLASS = {57R30 (57N10)},
  MRNUMBER = {682826},
MRREVIEWER = {Ulrich\ Hirsch},
       DOI = {10.1090/S0273-0979-1983-15089-9},
       URL = {https://doi.org/10.1090/S0273-0979-1983-15089-9},
}

@article {OSfourmanifolds,
    AUTHOR = {Ozsv\'ath, Peter and Szab\'o, Zolt\'an},
     TITLE = {Holomorphic triangles and invariants for smooth
              four-manifolds},
   JOURNAL = {Adv. Math.},
  FJOURNAL = {Advances in Mathematics},
    VOLUME = {202},
      YEAR = {2006},
    NUMBER = {2},
     PAGES = {326--400},
      ISSN = {0001-8708,1090-2082},
   MRCLASS = {57R58 (57M27)},
  MRNUMBER = {2222356},
       DOI = {10.1016/j.aim.2005.03.014},
       URL = {https://doi.org/10.1016/j.aim.2005.03.014},
}

@article {OSsymplectic,
    AUTHOR = {Ozsv\'ath, Peter and Szab\'o, Zolt\'an},
     TITLE = {Holomorphic triangle invariants and the topology of symplectic
              four-manifolds},
   JOURNAL = {Duke Math. J.},
  FJOURNAL = {Duke Mathematical Journal},
    VOLUME = {121},
      YEAR = {2004},
    NUMBER = {1},
     PAGES = {1--34},
      ISSN = {0012-7094,1547-7398},
   MRCLASS = {57R17 (53D35 57N13 57R57)},
  MRNUMBER = {2031164},
MRREVIEWER = {Stanislav\ Jabuka},
       DOI = {10.1215/S0012-7094-04-12111-6},
       URL = {https://doi.org/10.1215/S0012-7094-04-12111-6},
}

@article {liscamatic,
    AUTHOR = {Lisca, P. and Mati\'c, G.},
     TITLE = {Stein {$4$}-manifolds with boundary and contact structures},
      NOTE = {Symplectic, contact and low-dimensional topology (Athens, GA,
              1996)},
              
   JOURNAL = {Topology Appl.},
  FJOURNAL = {Topology and its Applications},
    VOLUME = {88},
      YEAR = {1998},
    NUMBER = {1-2},
     PAGES = {55--66},
      ISSN = {0166-8641,1879-3207},
   MRCLASS = {57R57 (32E10 57M50 57R15)},
  MRNUMBER = {1634563},
MRREVIEWER = {John\ B.\ Etnyre},
       DOI = {10.1016/S0166-8641(97)00198-3},
       URL = {https://doi.org/10.1016/S0166-8641(97)00198-3},
}

@incollection {cassonthree,
    AUTHOR = {Casson, Andrew J.},
     TITLE = {Three lectures on new-infinite constructions in
              {$4$}-dimensional manifolds},
 BOOKTITLE = {\`A{} la recherche de la topologie perdue},
    SERIES = {Progr. Math.},
    VOLUME = {62},
     PAGES = {201--244},
      NOTE = {With an appendix by L. Siebenmann},
 PUBLISHER = {Birkh\"auser Boston, Boston, MA},
      YEAR = {1986},
      ISBN = {0-8176-3329-4},
   MRCLASS = {57N13},
  MRNUMBER = {900253},
}

@article {heddenwatson,
    AUTHOR = {Hedden, Matthew and Watson, Liam},
     TITLE = {On the geography and botany of knot {F}loer homology},
   JOURNAL = {Selecta Math. (N.S.)},
  FJOURNAL = {Selecta Mathematica. New Series},
    VOLUME = {24},
      YEAR = {2018},
    NUMBER = {2},
     PAGES = {997--1037},
      ISSN = {1022-1824,1420-9020},
   MRCLASS = {57M27 (57R58)},
  MRNUMBER = {3782416},
MRREVIEWER = {David\ Shea\ Vela-Vick},
       DOI = {10.1007/s00029-017-0351-5},
       URL = {https://doi.org/10.1007/s00029-017-0351-5},
}

@article {wangcrossing,
    AUTHOR = {Wang, Joshua},
     TITLE = {The cosmetic crossing conjecture for split links},
   JOURNAL = {Geom. Topol.},
  FJOURNAL = {Geometry \& Topology},
    VOLUME = {26},
      YEAR = {2022},
    NUMBER = {7},
     PAGES = {2941--3053},
      ISSN = {1465-3060,1364-0380},
   MRCLASS = {57K10 (57K18)},
  MRNUMBER = {4540900},
MRREVIEWER = {Antonio\ Alfieri},
       DOI = {10.2140/gt.2022.26.2941},
       URL = {https://doi.org/10.2140/gt.2022.26.2941},
}

@article {gompfproper,
    AUTHOR = {Gompf, Robert E.},
     TITLE = {Topologically trivial proper 2-knots},
   JOURNAL = {Geom. Topol.},
  FJOURNAL = {Geometry \& Topology},
    VOLUME = {29},
      YEAR = {2025},
    NUMBER = {1},
     PAGES = {71--125},
      ISSN = {1465-3060,1364-0380},
   MRCLASS = {57K45 (57K40 57R40)},
  MRNUMBER = {4846638},
MRREVIEWER = {Inasa\ Nakamura},
       DOI = {10.2140/gt.2025.29.71},
       URL = {https://doi.org/10.2140/gt.2025.29.71},
}

@book {freedmanquinn,
    AUTHOR = {Freedman, Michael H. and Quinn, Frank},
     TITLE = {Topology of 4-manifolds},
    SERIES = {Princeton Mathematical Series},
    VOLUME = {39},
 PUBLISHER = {Princeton University Press, Princeton, NJ},
      YEAR = {1990},
     PAGES = {viii+259},
      ISBN = {0-691-08577-3},
   MRCLASS = {57N13 (57-02)},
  MRNUMBER = {1201584},
MRREVIEWER = {Ian\ Hambleton},
}

@inproceedings {freedmanICM,
    AUTHOR = {Freedman, Michael H.},
     TITLE = {The disk theorem for four-dimensional manifolds},
 BOOKTITLE = {Proceedings of the {I}nternational {C}ongress of
              {M}athematicians, {V}ol.\ 1, 2 ({W}arsaw, 1983)},
     PAGES = {647--663},
 PUBLISHER = {PWN, Warsaw},
      YEAR = {1984},
      ISBN = {83-01-05523-5},
   MRCLASS = {57N15 (57R65)},
  MRNUMBER = {804721},
MRREVIEWER = {Frank\ Quinn},
}

@article {gompfsteinisotopy,
    AUTHOR = {Gompf, Robert E.},
     TITLE = {Creating {S}tein surfaces by topological isotopy},
   JOURNAL = {J. Differential Geom.},
  FJOURNAL = {Journal of Differential Geometry},
    VOLUME = {125},
      YEAR = {2023},
    NUMBER = {1},
     PAGES = {121--171},
      ISSN = {0022-040X,1945-743X},
   MRCLASS = {32Q28 (57R52)},
  MRNUMBER = {4643804},
MRREVIEWER = {Rafael\ B.\ Andrist},
       DOI = {10.4310/jdg/1695236593},
       URL = {https://doi.org/10.4310/jdg/1695236593},
}

@article {siebenmann,
    AUTHOR = {Siebenmann, L. C.},
     TITLE = {Infinite simple homotopy types},
      NOTE = {Nederl. Akad. Wetensch. Proc. Ser. A {\bf 73}},
   JOURNAL = {Indag. Math.},
  FJOURNAL = {},
    VOLUME = {32},
      YEAR = {1970},
     PAGES = {479--495},
   MRCLASS = {55.40},
  MRNUMBER = {287542},
MRREVIEWER = {A.\ J.\ Sieradski},
}

@article {gompfsteinhandlebody,
    AUTHOR = {Gompf, Robert E.},
     TITLE = {Handlebody construction of {S}tein surfaces},
   JOURNAL = {Ann. of Math. (2)},
  FJOURNAL = {Annals of Mathematics. Second Series},
    VOLUME = {148},
      YEAR = {1998},
    NUMBER = {2},
     PAGES = {619--693},
      ISSN = {0003-486X,1939-8980},
   MRCLASS = {57R17 (32E10 57R65)},
  MRNUMBER = {1668563},
MRREVIEWER = {Selman\ Akbulut},
       DOI = {10.2307/121005},
       URL = {https://doi.org/10.2307/121005},
}

@article {gompfgenera,
    AUTHOR = {Gompf, Robert E.},
     TITLE = {Minimal genera of open 4-manifolds},
   JOURNAL = {Geom. Topol.},
  FJOURNAL = {Geometry \& Topology},
    VOLUME = {21},
      YEAR = {2017},
    NUMBER = {1},
     PAGES = {107--155},
      ISSN = {1465-3060,1364-0380},
   MRCLASS = {57R10 (32Q28)},
  MRNUMBER = {3608710},
MRREVIEWER = {Arkadiy\ Skopenkov},
       DOI = {10.2140/gt.2017.21.107},
       URL = {https://doi.org/10.2140/gt.2017.21.107},
}

@article{etnyrefillings,
author = {John B Etnyre},
title = {{On symplectic fillings}},
volume = {4},
journal = {Algebraic \& Geometric Topology},
number = {1},
publisher = {MSP},
pages = {73 -- 80},
year = {2004},
doi = {10.2140/agt.2004.4.73},
URL = {https://doi.org/10.2140/agt.2004.4.73}
}

@misc{knotinfo,
Author = {Livingston, Charles and Moore, Allison H.},
howpublished = {URL: \url{knotinfo.org}},
Month = {May},
Title = {{K}not{I}nfo: Table of Knot Invariants},
Year = {2016},
}
\Addresses
\end{document}